\documentclass[11pt, letterpaper, draft]{amsart}

\usepackage{amsfonts}
\usepackage{amssymb}
\usepackage{amsmath}
\usepackage{amsthm}
\usepackage{latexsym}
\usepackage{color}
\usepackage{comment}

\newtheorem{thm}{Theorem}[section]
\newtheorem{lemma}[thm]{Lemma}
\newtheorem{prop}[thm]{Proposition}
\newtheorem{cor}[thm]{Corollary}

\newtheorem*{thm*}{Theorem}
\theoremstyle{definition}
\newtheorem*{remark}{Remark}
\newtheorem{defi}[thm]{Definition}

\numberwithin{equation}{section}

\newcommand{\be}{\begin{equation}}
\newcommand{\ee}{\end{equation}}
\newcommand{\bee}{\begin{equation*}}
\newcommand{\eee}{\end{equation*}}
\newcommand{\bea}{\begin{eqnarray}}
\newcommand{\eea}{\end{eqnarray}}
\newcommand{\bs}{\begin{split}}
\newcommand{\es}{\end{split}}

\DeclareMathOperator{\Hyperg}{\mbox{ }_2 F_{1}}
\DeclareMathOperator{\divergence}{div}

\begin{document}

\title[Nonlocal ODE]
{ODE methods in non-local equations}

\author[W. Ao]{Weiwei Ao}

\address{Weiwei Ao
\hfill\break\indent
Wuhan University
\hfill\break\indent
Department of Mathematics and Statistics, Wuhan, 430072, PR China}
\email{wwao@whu.edu.cn}

\author[H. Chan]{Hardy Chan}

\address{Hardy Chan
\hfill\break\indent
ETH Z\"{u}rich,
\hfill\break\indent
Department of Mathematics, R\"{a}mistrasse 101, 8092 Z\"{u}rich, Switzerland}
\email{hardy.chan@math.ethz.ch}

\author[A. DelaTorre]{Azahara DelaTorre}

\address{Azahara DelaTorre
\hfill\break\indent
Albert-Ludwigs-Universit\"{a}t Freiburg
\hfill\break\indent
Mathematisches Institut, Ernst-Zermelo-Str.1,  D-79104 Freiburg (Breisgau), Germany.}
\email{azahara.de.la.torre@math.uni-freiburg.de}

\author[M. Fontelos]{Marco A. Fontelos}

\address{Marco A. Fontelos
\hfill\break\indent
ICMAT,
\hfill\break\indent
Campus de Cantoblanco, UAM, 28049 Madrid, Spain}
\email{marco.fontelos@icmat.es}

\author[M.d.M. Gonz\'alez]{Mar\'ia del Mar Gonz\'alez}

\address{Mar\'ia del Mar Gonz\'alez
\hfill\break\indent
Universidad Aut\'onoma de Madrid
\hfill\break\indent
Departamento de Matem\'aticas, Campus de Cantoblanco, 28049 Madrid, Spain}
\email{mariamar.gonzalezn@uam.es}

\author[J. Wei]{Juncheng Wei}
\address{Juncheng Wei
\hfill\break\indent
University of British Columbia
\hfill\break\indent
 Department of Mathematics, Vancouver, BC V6T1Z2, Canada} \email{jcwei@math.ubc.ca}

\maketitle

\begin{abstract}
Non-local equations cannot be treated using classical ODE theorems. Nevertheless, several new methods have been introduced in the non-local gluing scheme of our previous article \cite{Ao-Chan-DelaTorre-Fontelos-Gonzalez-Wei}; we survey and improve those, and present new applications as well.

First, from the explicit symbol of the conformal fractional Laplacian, a variation of constants formula is obtained for fractional Hardy operators. We thus develop, in addition to a suitable extension in the spirit of Caffarelli--Silvestre, an equivalent formulation as an infinite system of second order constant coefficient ODEs. Classical ODE quantities like the Hamiltonian and Wro\'{n}skian may then be utilized.

As applications, we obtain a Frobenius theorem and establish new Poho\v{z}aev identities. We also give a detailed proof for the non-degeneracy of the fast-decay singular solution  of the fractional Lane--Emden equation.


\end{abstract}

\section{Introduction}

Let $\gamma\in(0,1)$. We consider radially symmetric solutions  the fractional Laplacian equation
\begin{equation}\label{problem10}
(-\Delta)^\gamma u=Au^p\quad\text{in }\mathbb R^n\setminus\{0\},
\end{equation}
with an isolated singularity at the origin. Here $$p\in\left(\frac{n}{n-2\gamma},\frac{n+2\gamma}{n-2\gamma}\right],$$
and the constant $A:=A_{n,p,\gamma}(>0)$ is chosen so that $u_0(r)=r^{-\frac{2\gamma}{p-1}}$ is a singular solution to the equation. Note that this the exact growth rate around the origin of any other solution with non-removable singularity according to \cite{Ao-Chan-DelaTorre-Fontelos-Gonzalez-Wei, Caffarelli-Jin-Sire-Xiong,Chen-Quaas}.

Note that for $$p=\frac{n+2\gamma}{n-2\gamma}$$ the problem is critical for the Sobolev embedding $W^{\gamma,2}\hookrightarrow L^{\frac{2n}{n-2\gamma}}$. In addition, for this choice of nonlinearity, the equation has good conformal properties and, indeed, in conformal geometry it is known as the fractional Yamabe problem. In this case the constant $A$ coincides with the Hardy constant $\Lambda_{n,\gamma}$ given in \eqref{Hardy-constant}.

 There is an extensive literature on the fractional Yamabe problem by now. See \cite{Gonzalez-Qing,Gonzalez-Wang,Kim-Musso-Wei,Mayer-Ndiaye} for the smooth case, \cite{DelaTorre-Gonzalez,DelaTorre-delPino-Gonzalez-Wei,Ao-DelaTorre-Gonzalez-Wei,Ao-Gonzalez-Sire} in the presence of isolated singularities, and \cite{Gonzalez-Mazzeo-Sire,Ao-Chan-Gonzalez-Wei,Ao-Chan-DelaTorre-Fontelos-Gonzalez-Wei} when the singularities are not isolated but a higher dimensional set.

In this paper we take the analytical point of view and study several non-local ODE that are related to  problem \eqref{problem10}, presenting both survey and new results, in the hope that this paper serves as a guide for non-local ODE. A non-local equation such as \eqref{problem10} for  radially symmetric solutions $u=u(r)$, $r=|x|$, requires different techniques than regular ODE. For instance, existence and uniqueness theorems are not available in general, so one cannot reduce it to the study of a phase portrait. Moreover, the asymptotic behavior as $r\to 0$ or $r\to \infty$ is not clear either.

However, we will show that, in some sense, \eqref{problem10} behaves closely to its local counterpart (the case $\gamma=1$), which is given by the second order ODE
\begin{equation*}
\partial_{rr}u+\frac{n-1}{r}\partial_r u=Au^p.
\end{equation*}
%
In particular, for the survey part we will extract many results for non-local ODEs from the long paper \cite{Ao-Chan-DelaTorre-Fontelos-Gonzalez-Wei} but without many of the technicalities. However, from the time since \cite{Ao-Chan-DelaTorre-Fontelos-Gonzalez-Wei} first appeared, some of the proofs have been simplified; we present those in detail.

The main underlying idea, which was not fully exploited in \cite{Ao-Chan-DelaTorre-Fontelos-Gonzalez-Wei}, is to write problem \eqref{problem10} as an infinite dimensional ODE system. Each equation in the system is a standard second order ODE, the non-locality appears in the coupling of the right hand sides (see Corollary \ref{cor:ODEsystem}). The advantage of this formulation comes from the fact that, even though we started with a non-local ODE, we can still use a number of the standard results, as long as one takes care of this coupling.
For instance, we will be able to write the indicial roots for the system and a Wro\'{n}skian-type quantity which will be useful in the uniqueness proofs. Other applications include  novel Poho\v{z}aev-type identities. We also hope that this paper serves as a complement to the elliptic theory of differential edge operators from \cite{Mazzeo:edge,Mazzeo:edge2}. Further use could include more general semi-linear equations such as the one in \cite{Chan-Gonzalez-Huang-Mainini-Volzone}, but this is yet to be explored.

\medskip

Let us summarize our results. First, in Section \ref{section:Hamiltonian}, we consider existence theorems for \eqref{problem10}, both in the critical and subcritical case. We show that the change of variable
\begin{equation}\label{change-variable}
r=e^{-t}, \quad u(r)=r^{-\frac{2\gamma}{p-1}}v(-\log r)
\end{equation}
transforms \eqref{problem10} into the non-local equation of the form
\begin{equation}\label{problem11-1.3}
\int_{\mathbb R} \tilde{\mathcal K}(t-t')[v(t)-v(t')]\,dt'+Av(t)=Av(t)^p,\quad v=v(t), \quad t\in\mathbb R,
\end{equation}
for some singular kernel satisfying $\tilde{\mathcal K}(t)\sim |t|^{-1-2\gamma}$ as $|t|\to 0$. The advantage of \eqref{problem11-1.3} over the original \eqref{problem10} is that in the new variables the problem becomes \emph{autonomous} in some sense. Thus, even if we cannot plot it, one expects some kind of phase-portrait. Indeed, we show the existence of a monotone quantity (a \emph{Hamiltonian}) similar to those of \cite{CabreSola-Morales,Cabre-Sire1,Frank-Lenzmann-Silvestre} which, in the particular case $p=\frac{n+2\gamma}{n-2\gamma}$ is conserved along the $t$-flow.

The proofs have its origin in conformal geometry and, in particular, we give an interpretation of the change of variable \eqref{change-variable} in terms of the conformal fractional Laplacian on the cylinder. We also provide an extension problem for \eqref{problem11-1.3} in the spirit of the well known extension for the fractional Laplacian (\cite{Caffarelli-Silvestre,Stinga-Torrea} and many others). Note, however, that our particular extension has its origins in scattering theory on conformally compact Einstein manifolds (see \cite{Graham-Zworski,Chang-Gonzalez,Case-Chang} and the survey \cite{Gonzalez:survey}, for instance) and it does not produce fractional powers of operators, but their conformal  versions. This is the content of Section \ref{section:conformal}.

\medskip

In addition, for $p$ subcritical, this is, $$p\in\Big(\frac{n}{n-2\gamma},\frac{n+2\gamma}{n-2\gamma}\Big),$$
we will also consider radial solutions to the linearized problem around a
certain
solution $u_*$. The resulting equation may be written as
\begin{equation*}
(-\Delta)^\gamma \phi-pAu_*^{p-1}\phi=0,\quad \phi=\phi(r).
\end{equation*}
Defining the radially symmetric potential $\mathcal V_*(r):=pAr^{2\gamma}u_*^{p-1}$, this equation is equivalent to
\begin{equation}\label{linearization-introduction}
L_*\phi:=(-\Delta)^\gamma \phi-\frac{\mathcal V_*(r)}{r^{2\gamma}}\phi=0,\quad \phi=\phi(r).
\end{equation}
Note that
\begin{equation}
\label{behavior-potential-introduction}
\mathcal V_*(r)\to\kappa\quad\text{as} \quad r\to 0
\end{equation}
for the positive constant
\begin{equation*}\label{kappa}\kappa=pA.
\end{equation*}
Therefore to understand operators with critical Hardy potentials such as $L_*$ we will need to consider first
the constant coefficient operator
\begin{equation*}
L_\kappa:=(-\Delta)^\gamma -\frac{\kappa}{r^{2\gamma}}.
\end{equation*}
The fractional Hardy inequality (\cite{Prehistory-Hardy,Kovalenko-Perelmuter-Semenov,Herbst,Yafaev,Beckner:Pitts,Frank-Lieb-Seiringer:Hardy}) asserts the non-negativity of such operator up to $\kappa=\Lambda_{n,\gamma}$, hence whenever
\begin{equation}\label{eq:stability-threshold}
pA \leq \Lambda_{n,\gamma},
\end{equation}
and this distinguishes the stable/unstable cases (see Definition \ref{defi:stable}).

Nonlinear Schr\"odinger equations with fractional Laplacian have received a lot of attention recently (see, for instance, \cite{Felmer-Quaas-Tan,Davila-DelPino-Wei,Ao-Chan-Gonzalez-Wei:supercritical}), while ground state solutions for non-local problems have been  considered in the papers \cite{Frank-Lenzmann,Frank-Lenzmann-Silvestre}.

However, mapping properties for a linear operator such as $L_*$ from \eqref{linearization-introduction}, or even $L_\kappa$, had been mostly open until the publication of \cite{Ao-Chan-DelaTorre-Fontelos-Gonzalez-Wei}. See also the related paper \cite{Frank-Merz-Siedentop}, which deals with mapping properties of powers of this operator in homogeneous Sobolev spaces, together with some applications from \cite{Frank-Merz-Siedentop-Simon}, where they prove the Scott conjecture for large atoms taking into account relativistic effects near the nucleus.

 One of the cornerstones in \cite{Ao-Chan-DelaTorre-Fontelos-Gonzalez-Wei} is to write a Green's function for the constant coefficient operator $L_\kappa$
 in suitable weighted spaces. While invertibility for $L_\kappa$ in terms of the behavior of its right hand side had been considered in \cite{Abdellaoui-Medina-Peral-Primo}, here we go further and calculate the indicial roots of the problem to characterize invertibility precisely.  This is done by writing a variation of constants formula to produce solutions to $L_\kappa\phi=h$ from elements in the kernel $L_\kappa \phi=0$. In particular, such $\phi$ is governed by the indicial roots of the equation. However, in contrast to the local case where a second order ODE only has two indicial roots, here we find an infinite number of them and, moreover, the solution is not just a combination of two linearly independent solutions of the homogeneous problem, but an infinite sum. We summarize those results in Section \ref{section:Hardy}. We also present, in full detail, simplified proofs of the original statements.

One obtains, as a consequence, a Frobenious type theorem which yields a precise asymptotic expansion for solutions to \eqref{linearization-introduction} in terms of the asymptotics of the potential as $r\to 0$.  Indeed,  recall that we have that \eqref{behavior-potential-introduction}, so we can use what we know about $L_\kappa$ in order to obtain information about $L_*$. In particular, we find the indicial roots of $L_*$ both as $r\to 0$ and as $r\to \infty$.

Next, we move on to original research. Section \ref{section:non-degeneracy} is a combination of new and known results. There we give full account of non-degeneracy of equation \eqref{problem10} for the particular solution $u_*$, this is, we provide a characterization of the kernel of the linearized operator $L_*$, both in the stable and in the unstable cases.

One of the main contributions of this paper is the introduction of a new Wro\'{n}skian quantity \eqref{Wronskian} for a non-local ODE such as \eqref{linearization-introduction}, that allows to compare any two solutions, and plays the role of the usual Wro\'{n}skian $W=w_1'w_2-w_1w_2'$ for a second order linear ODE. While this quantity is close in spirit to that of \cite{CabreSola-Morales,Cabre-Sire1,Frank-Lenzmann-Silvestre}, ours seems to adapt better to an autonomous non-local ODE.

Using similar techniques for the non-linear problem, we also provide new Poho\v{z}aev identities: Proposition \ref{prop:Pohozaev1} (for the extension problem) and Proposition \ref{prop:Pohozaev2} (for the non-local problem formulated as a coupled ODE system).   The underlying idea is that, switching from the radial variable $r$ to the logarithmic $t$ variable, and from the standard fractional Laplacian $(-\Delta)^\gamma$ to the conformal fractional Laplacian $P_\gamma$, we are able to find conserved quantities since the resulting ODE is, though non-local, autonomous.

\medskip

Some comments about notation:
\begin{itemize}
\item The Hardy constant is given by
\begin{equation}\label{Hardy-constant}
\Lambda_{n,\gamma}=2^{2\gamma}\left(\frac{\Gamma(\frac{n+2\gamma}{4})}{\Gamma(\frac{n-2\gamma}{4})}\right)^2,
\end{equation}
where $\Gamma$ is the ordinary Gamma function.
\item We write $f_1\asymp f_2$ if the two (positive) functions satisfy $C^{-1}f_1 \leq f_2 \leq C f_1$ for some positive constant $C$.
\item While functions that live on $\mathbb R^n$ are represented by lowercase letters, their corresponding extension to $\mathbb R^{n+1}_+$ will be denoted by the same letter in capitals.
\item $\mathbb S^{n}$ is the unit sphere with its canonical metric.
\item $\Hyperg$ is the standard hypergeometric function.
\end{itemize}

\section{The conformal fractional Laplacian on the cylinder}\label{section:conformal}

Most of the results here can be found in \cite{Ao-Chan-DelaTorre-Fontelos-Gonzalez-Wei}. We review the construction of the fractional Laplacian on the cylinder $\mathbb R \times \mathbb S^{n-1}$, both by Fourier methods and by constructing an extension problem to one more dimension.

\subsection{Conjugation}

Let us look first at the critical power, and consider the equation
\begin{equation}\label{problem-critical}
(-\Delta)^\gamma u=\Lambda_{n,\gamma}u^\frac{n+2\gamma}{n-2\gamma}\quad\text{in }\mathbb R^n\setminus\{0\}.
\end{equation}
It is well known \cite{Caffarelli-Jin-Sire-Xiong} that non-removable singularities must be, as $r\to 0$, of the form
\begin{equation}\label{uw}
u=r^{-\frac{n-2\gamma}{2}}w,
\end{equation}
for a bounded function $w$.

We use conformal geometry to rewrite the fractional Laplacian operator $(-\Delta)^\gamma$ on $\mathbb R^n$ in radial coordinates $r>0$, $\theta\in\mathbb S^{n-1}$. The Euclidean metric in polar coordinates is given by $|dx|^2=dr^2+r^2d\theta^2$, where we denote $d\theta^2$ for the metric on the standard sphere $\mathbb S^{n-1}$. Now set the new variable
\begin{equation*}
t=-\log r\in\mathbb R
\end{equation*}
 and consider the cylinder $M=\mathbb R\times \mathbb S^{n-1}$, with the metric given by
$$g_c:=\frac{1}{r^2} |dx|^2=\frac{1}{r^2}(dr^2+r^2d\theta^2)=dt^2+d\theta^2,$$
which is conformal to the original Euclidean one.

The conformal fractional Laplacian on the cylinder, denoted by  $P_\gamma$, is defined in \cite{DelaTorre-Gonzalez}. We will not present the full construction here, just mention that its conformal property implies
\begin{equation}\label{conformal-property}
P_\gamma w= r^{\frac{n+2\gamma}{2}} (-\Delta)^\gamma (r^{-\frac{n-2\gamma}{2}}w),
\end{equation}
and thus, if we define $w$ by \eqref{uw}, in the new variables $w=w(t,\theta)$, then the original equation \eqref{problem-critical} transforms into
\begin{equation*}
P_\gamma w=\Lambda_{n,\gamma}w^\frac{n+2\gamma}{n-2\gamma},\quad t\in\mathbb R, \ \theta \in\mathbb S^{n-1}
\end{equation*}
for a smooth solution $w$. Therefore we have shifted the singularity from the solution to the metric.

On the contrary, equation  \eqref{problem10} for a subcritical power
$$p\in \Big(\frac{n}{n-2\gamma},\frac{n+2\gamma}{n-2\gamma}\Big)$$
does not have good conformal properties. Still, given
 $u\in \mathcal C^{\infty}(\mathbb R^{n}\setminus\{0\})$, we can consider
\begin{equation*}
u=r^{-\frac{n-2\gamma}{2}}w=r^{-\frac{2\gamma}{p-1}}v,\quad r=e^{-t},
\end{equation*}
and define the conjugate operator
\begin{equation*}\label{tilde-P}
\tilde P_\gamma(v):=
r^{Q_0}P_\gamma\big(r^{-Q_0}v\big)
=r^{\frac{2\gamma}{p-1}p}(-\Delta_{\mathbb R^n})^\gamma u.
\end{equation*}
where we have defined the constant
\begin{equation*}
Q_0:=-\tfrac{n-2\gamma}{2}+\tfrac{2\gamma}{p-1}>0.
\end{equation*}
The advantage of working with $\tilde P_\gamma$ is that problem \eqref{problem10} is equivalent to
\begin{equation*}
\tilde P_\gamma v =Av^{p},\quad t\in\mathbb R, \ \theta\in\mathbb S^{n-1},
\end{equation*}
for some $v=v(t,\theta)$ smooth.

\subsection{The operator as a singular integral}

Both $P_\gamma$ and $\tilde P_\gamma$ on $\mathbb R\times\mathbb S^{n-1}$ have been well studied. Note that $Q_0=0$ for the critical value of $p$, so that $\tilde P_\gamma$ reduces to the original $P_\gamma$. In the following, we will only present the results for $\tilde P_\gamma$ but, of course, they are true also for $P_\gamma$ (simply substitute $Q_0=0$ in the statement).

Consider the spherical harmonic decomposition for $\mathbb S^{n-1}$. With some abuse of notation, let $\mu_m$, $m=0,1,2,\dots$ be the eigenvalues of $\Delta_{\mathbb S^{n-1}}$, repeated according to multiplicity (this is, $\mu_0=0$, $\mu_1,\ldots,\mu_n=n-1$, \dots). Then any function $v$ on $\mathbb R\times \mathbb S^{n-1}$ may be decomposed as
\begin{equation*}
v(t,\theta)=\sum_{m} v_m(t) E_m(\theta),
\end{equation*}
where $\{E_m(\theta)\}$ is the corresponding basis of eigenfunctions.

The operator $\tilde P_\gamma$ diagonalizes under such eigenspace decomposition, and moreover, it is possible to calculate the Fourier symbol of each projection. Let
\begin{equation*}\label{fourier}
\hat{v}(\xi)=\frac{1}{\sqrt{2\pi}}\int_{\mathbb R}e^{-i\xi \cdot t} v(t)\,dt
\end{equation*}
be our normalization for the one-dimensional Fourier transform. Let us define also
\begin{equation*}\label{Am}
A_m=\tfrac{1}{2}+\tfrac{\gamma}{2}
+\tfrac{1}{2}\sqrt{\big(\tfrac{n}{2}-1\big)^2+\mu_m},\quad B_m=\tfrac{1}{2}-\tfrac{\gamma}{2}+\tfrac{1}{2}\sqrt{\big(\tfrac{n}{2}-1\big)^2+\mu_m}.
\end{equation*}

\begin{prop}[\cite{DelaTorre-Gonzalez,Ao-Chan-DelaTorre-Fontelos-Gonzalez-Wei}]\label{prop:symbol}
 Fix $\gamma\in (0,\tfrac{n}{2})$ and let $\tilde{P}^m_{\gamma}$ be the projection of the operator $\tilde{P}_\gamma$ over each eigenspace $\langle E_m\rangle$. Then
$$\widehat{\tilde{P}_\gamma^{(m)} (v_m)}(\xi)
=\tilde{\Theta}^{(m)}_\gamma(\xi) \,\widehat{v_m}(\xi),$$
and this Fourier symbol is given by
\begin{equation*}\label{symbolV}
\tilde{\Theta}^{(m)}_{\gamma}(\xi)=2^{2\gamma}\frac{\Gamma \big(A_m
+\frac{1}{2}(Q_0+\xi i)\big)\Gamma \big(A_m
-\frac{1}{2}(Q_0+\xi i)\big)}
{\Gamma \big(B_m
+\frac{1}{2}(Q_0+\xi i)\big)\Gamma \big(B_m
-\frac{1}{2}(Q_0+\xi i)\big)}.
\end{equation*}
Moreover,
\begin{equation*}
\tilde P_\gamma^{(m)}(v_m)(t)=\int_{\mathbb R} \tilde{\mathcal K}_m(t-t')[v_m(t)-v_m(t')]\,dt'+A_{n,p,\gamma}v_m(t),
\end{equation*}
for a convolution kernel $\tilde{\mathcal K}_m$ on $\mathbb R$
with the asymptotic behavior
\begin{equation*}
\tilde{\mathcal K}_m(t)\asymp
\begin{cases}
 |t|^{-1-2\gamma} &\mbox{ as }|t|\to 0,\\
e^{-\big(1+\gamma+\sqrt{(\frac{N-2}{2})^2+\mu_m}+Q_0\big)t} &\mbox{ as }t \to +\infty,\\
e^{\big(1+\gamma+\sqrt{(\frac{N-2}{2})^2+\mu_m}-Q_0\big)t} &\mbox{ as }t \to -\infty.
\end{cases}
\end{equation*}
In the particular case that $m=0$,
\begin{equation*}\label{kernel-K}
\tilde{\mathcal K}_0(t)=
c\,e^{-(\frac{2\gamma}{p-1}-\frac{n-2\gamma}{2})t}e^{-\frac{n+2\gamma}{2}|t|}\Hyperg\big( \tfrac{n+2\gamma}{2},1+\gamma;\tfrac{n}{2};e^{-2|t|}\big).
\end{equation*}
\end{prop}

\subsection{The extension problem}

For $\gamma\in(0,1)$, this operator can also be understood as the Dirichlet-to-Neumann operator for an extension problem in the spirit of \cite{Caffarelli-Silvestre,Chang-Gonzalez,Case-Chang}. Before we do that, we need to introduce some notation. Define
\begin{equation*}\tilde d_\gamma =-
\frac{2^{2\gamma-1}\Gamma(\gamma)}{\gamma\Gamma(-\gamma)},\quad
\alpha
=\frac{\Gamma(\frac{n}{2})\Gamma(\gamma)}
{\Gamma\big(\gamma+\frac{\gamma}{p-1}\big)
\Gamma\big(\frac{n}{2}-\frac{\gamma}{p-1}\big)}.
\end{equation*}

Take the metric on the extension manifold $X^{n+1}=(0,2)\times \mathbb R\times\mathbb S^{n-1}$ with coordinates $R\in(0,2)$, $t\in\mathbb R$, $\theta\in\mathbb S^{n-1}$
with standard hyperbolic metric
\begin{equation*}
\bar g=dR^2+\Big(1+\tfrac{R^2}{4}\Big)^2 dt^2+\Big(1-\tfrac{R^2}{4}\Big)^2 d\theta^2.
\end{equation*}
Note that the apparent singularity at $R=2$ has the same behavior  as the origin in polar coordinates. This fact will be implicitly assumed in the following exposition without further mention.

The boundary of $X^{n+1}$ (actually, its conformal infinity) is given by $\{R=0\}$, and it coincides precisely the cylinder $M^n=\mathbb R\times \mathbb S^{n-1}$, with its canonical metric
$g_c=dt^2+d\theta^2$.

 Now we make the change of variables from the coordinate $R$ to
\begin{equation*}\label{rho*}
\rho(R)=
\left[\alpha^{-1}\big(\tfrac{4R}{4+R^2}\big)^{\tfrac{n-2\gamma}{2}}
\Hyperg\Big(\tfrac{\gamma}{p-1},\tfrac{n-2\gamma}{2}-\tfrac{\gamma}{p-1};
\tfrac{n}{2};\big(\tfrac{4-R^2}{4+R^2}\big)^2\Big)\right]^{2/(n-2\gamma)},
\quad R\in(0,2).
\end{equation*}
The function $\rho$ is known as the special (or adapted) defining function. It is strictly monotone with respect to $R$, which implies that we can write $R=R(\rho)$ even if we do not have a precise formula and, in particular,
$\rho\in(0,\rho_0)$ for
\begin{equation*}\label{rho*0}
\rho_0:=\rho(2)=\alpha^{-\frac{2}{n-2\gamma}}.
\end{equation*}
Moreover, it has the asymptotic expansion near the conformal infinity
\begin{equation*}
\label{asymptotics-rho*}\rho(R)=R\left[ 1+O(R^{2\gamma})\right].
\end{equation*}
In the new manifold $X^*=(0,\rho_0)\times\mathbb R\times\mathbb S^{n-1}$ consider the metric $\bar g^*:=(\frac{\rho}{R})^2\bar g$. This metric satisfies
$$\bar g^*=d\rho^2(1+O(\rho^{2\gamma}))+g_c(1+O(\rho^{2\gamma})).$$

\begin{prop}[\cite{Ao-Chan-DelaTorre-Fontelos-Gonzalez-Wei}]\label{prop:divV*}
Let $v$ be a smooth function on $M=\mathbb R\times\mathbb S^{n-1}$.
The extension problem
\begin{equation}\label{divV*}
\left\{\begin{array}{@{}r@{}l@{}l}
 -\divergence_{\bar{g}^*}(\rho^{1-2\gamma}\nabla_{\bar{g}^*}V)
-\rho^{-(1+2\gamma)}\left(\tfrac{4R}{4+R^2}\right)^{2}2Q_0\,\partial_t V&\,=0\quad &\text{in } (X,\bar g^*), \medskip\\
V|_{\rho=0}&\,=v\quad &\text{on }M,
\end{array}\right.
\end{equation}
has a unique solution $V$. Moreover, for its Neumann data,
\begin{equation*}\label{Neumann-V*}
\tilde P_\gamma(v)=-\tilde d_\gamma \lim_{\rho \to 0}\rho^{1-2\gamma} \partial_{\rho} V+A_{n,p,\gamma}v.
\end{equation*}
\end{prop}

\begin{remark}[\cite{DelaTorre-Gonzalez}]
If $v$ is a radial function on $M$, this is, $v=v(t)$, then the first equation in \eqref{divV*} decouples to \begin{equation}\label{equation-extension}
\partial_{\rho}\left(e_1(\rho)\rho^{1-2\gamma}\partial_{\rho}V\right)
+e_2(\rho)\rho^{1-2\gamma}\partial_{tt}V=-Q_0 F(\rho)\,\partial_t V,
\end{equation}
for some $F(\rho)\geq 0$ and some continuous functions $e_l(\rho)$ which are smooth for $\rho\in(0,\rho_0)$ and that satisfy
$$\displaystyle\lim_{\rho\to 0}e_l(\rho)=1, \quad l=1,2.$$

\end{remark}

For convenience of the reader, let us particularize this result for $Q_0=0$ in order to give a characterization for $P_\gamma w$ if $w=w(t)$:

\begin{prop}[\cite{DelaTorre-Gonzalez}]\label{prop:extension}
Let $w=w(t)$ be a smooth, radially symmetric  function on $M=\mathbb R\times\mathbb S^{n-1}$.
The extension problem
\begin{equation*}
\left\{\begin{array}{@{}r@{}l@{}l}
 \partial_{\rho}\left(e_1(\rho)\rho^{1-2\gamma}\partial_{\rho}W\right)
+e_2(\rho)\rho^{1-2\gamma}\partial_{tt}W&\,=0\quad &\text{in } \rho\in(0,\rho_0),t\in\mathbb R, \medskip\\
W|_{\rho=0}&\,=w\quad &\text{on }M,
\end{array}\right.
\end{equation*}
has a unique solution $W=W(\rho,t)$. Moreover, for its Neumann data,
\begin{equation*}
P_\gamma(w)=P_\gamma^{(0)}w=-\tilde d_\gamma \lim_{\rho \to 0}\rho^{1-2\gamma} \partial_{\rho} W+\Lambda_{n,\gamma}w.
\end{equation*}
\end{prop}

\section{Non-local ODE: existence and Hamiltonian identities}\label{section:Hamiltonian}

In the following, we fix $\gamma\in(0,1)$. We consider radially symmetric solutions $u=u(r)$ to the non-linear problem
\begin{equation*}
(-\Delta)^\gamma u=cu^p \quad\text{in }\mathbb R^n\setminus\{0\}.
\end{equation*}

\subsection{The critical case}

For this part, set
$$p=\frac{n+2\gamma}{n-2\gamma}.$$
Let $u=u(r)$ be a radially symmetric solution to
\begin{equation}\label{problem111}
(-\Delta)^\gamma u=\Lambda_{n,\gamma}u^{\frac{n+2\gamma}{n-2\gamma}}, \quad\text{in }\mathbb R^n\setminus\{0\},
\end{equation}
and set
\begin{equation*}
u=r^{-\frac{n-2\gamma}{2}}w,\quad r=e^{-t},
\end{equation*}
then from the results in Section \ref{section:conformal}, we know that this problem is equivalent to
\begin{equation}\label{problem1111}
P_\gamma^{(0)} w =\Lambda_{n,\gamma}w^{\frac{n+2\gamma}{n-2\gamma}}, \quad w=w(t),\quad t\in\mathbb R.
\end{equation}

The advantage of the $t$ variable over the original $r$ is that one can show the existence of a Hamiltonian  similar to the monotone quantities from \cite{Cabre-Sire1,Frank-Lenzmann-Silvestre}. However, in our case, it is conserved along the $t$-flow:

\begin{thm}[\cite{Ao-Chan-DelaTorre-Fontelos-Gonzalez-Wei}]\label{thctthamiltonian}
Let $w=w(t)$ be a solution to \eqref{problem1111} and set $W$ its extension from Proposition \ref{prop:extension}.
Then, the Hamiltonian quantity
\begin{equation}\label{Hamiltonian}
\begin{split}
H_\gamma[W](t)=&\frac{\Lambda_{n,\gamma}}{\tilde{d}_{\gamma}}
\left(-\frac{1}{2}w^2+\frac{1}{p+1}w^{p+1}\right)\\
&+\frac{1}{2}\int_{0}^{\rho_0}  {\rho}^{1-2\gamma}\left\{-e'_1(\rho)(\partial_{\rho}W)^2
+e'_2(\rho)(\partial_t W)^2\right\}\,d\rho
\end{split}
\end{equation}
is constant along trajectories. Here $e'_l(\rho)$, $l=1,2$, are two smooth positive functions in $(0,\rho_0)$, continuous up to the boundary, that satisfy
$\displaystyle\lim_{\rho\to 0}e'_l(\rho)=1$.
\end{thm}

The fact that such a Hamiltonian quantity exists
 suggests that a non-local ODE should have a similar behavior as in the local second-order case, where one can draw a phase portrait. However, one cannot use standard ODE theory to prove existence and uniqueness of solutions.

In any case, we have two types of solutions to \eqref{problem1111} in addition to the constant solution $w\equiv 1$: first we find an explicit homoclinic, corresponding to the standard bubble
\begin{equation*}
u_\infty(r)=c\left(\frac{1}{1+r^2}\right)^{\frac{n-2\gamma}{2}}.
\end{equation*}
Indeed:

\begin{prop}[\cite{DelaTorre-Gonzalez}]\label{bubble}
The positive function
\begin{equation*}
w_\infty(t)=C(\cosh t)^{-\frac{n-2\gamma}{2}},\quad \text{ for}\quad C=
\left(\Lambda_{n,\gamma}\frac{\Gamma(\frac{n}{2}-\gamma)}{\Gamma(\frac{n}{2}+\gamma)}\right)^{-\frac{n-2\gamma}
{4\gamma}}>1,
\end{equation*}
is a smooth solution to \eqref{problem1111}.
\end{prop}

Second, we have periodic solutions (these are known as \emph{Delaunay} solutions). The proof is  variational and we refer to \cite{DelaTorre-delPino-Gonzalez-Wei} for further details:

\begin{thm}[\cite{DelaTorre-delPino-Gonzalez-Wei}]
Let $n>2+2\gamma$. There exists $L_0$ (the minimal period) such that for any $L>L_0$, there exists a periodic solution $w_L=w_L(t)$ to \eqref{problem1111} satisfying $w_L(t+L)=w_L(t)$.
\end{thm}

\subsection{The subcritical problem}\label{subsection:subcritical}

Now we consider the subcritical problem
\begin{equation*}
p\in\Big(\frac{n}{n-2\gamma},\frac{n+2\gamma}{n-2\gamma}\Big).
\end{equation*}
Let $u=u(r)$ be a radially symmetric solution to
\begin{equation}\label{problem}
(-\Delta)^\gamma u=Au^{p} \quad\text{in }\mathbb R^n,
\end{equation}
and set
\begin{equation*}
u=r^{-\frac{2\gamma}{p-1}}v,\quad r=e^{-t},
\end{equation*}
then \eqref{problem} is equivalent to
\begin{equation}\label{problem11}
\tilde P_\gamma^{(0)} v =Av^{p},
\end{equation}
for some $v=v(t)$, $t\in\mathbb R$.

\begin{thm}\cite{Ao-Chan-DelaTorre-Fontelos-Gonzalez-Wei}
Let $v=v(t)$ be a solution to \eqref{problem11} and set $V$ its extension from Proposition \ref{prop:divV*}.
Then, the Hamiltonian quantity
$H_\gamma[V](t)$ given in \eqref{Hamiltonian} is non-increasing in $t$.
\end{thm}

The study of \eqref{problem} (or equivalently, \eqref{problem11}) greatly depends on its linearized equation around the radial singular solution $u_0(r)=r^{-\frac{2\gamma}{p-1}}$, which involves the Hardy operator
\begin{equation*}
\label{L-kappa}L_\kappa:=(-\Delta)^\gamma-\frac{\kappa}{r^{2\gamma}},\quad \kappa=pA.
\end{equation*}
To study stability we need to understand the value of such $\kappa$. Note that the solution of \eqref{eq:stability-threshold} is completely characterized in \cite{Luo-Wei-Zou}. In particular, on the interval $\left(\frac{n}{n-2\gamma},\frac{n+2\gamma}{n-2\gamma}\right)$, \eqref{eq:stability-threshold} holds if and only if $p\leq p_1$, where $p_1$ is the unique real root of the equation $pA=\Lambda_{n,\gamma}$ on this interval.\footnote{
Remark that $A$ is a quotient of Gamma functions depending on $p$, while the Hardy constant depends only on $n$ and $\gamma$.
}
\begin{defi}\label{defi:stable}
We say that one is in the \emph{stable case} if
\[
p\in\left(\tfrac{n}{n-2\gamma},p_1\right],
\]
where the singular solution $r^{-\frac{2\gamma}{p-1}}$ is \emph{stable}. Similarly, one is in the \emph{unstable case} if
\[
p\in\left(p_1,\tfrac{n+2\gamma}{n-2\gamma}\right),
\]
where $r^{-\frac{2\gamma}{p-1}}$ is \emph{unstable}.
\end{defi}

\medskip

Existence of a fast decaying solution has been proved, in the stable case, in \cite{Ao-Chan-Gonzalez-Wei} and, in the unstable case, in \cite{Ao-Chan-DelaTorre-Fontelos-Gonzalez-Wei}. We summarize these results in the following theorem:

\begin{thm}[\cite{Ao-Chan-DelaTorre-Fontelos-Gonzalez-Wei}]\label{existence}
For any $\epsilon\in(0,\infty)$ there exists a fast-decaying, radially symmetric, entire singular solution $u_\epsilon$ of \eqref{problem} such that
\begin{equation*}
u_\epsilon(r)
=(1+o(1))
\begin{cases}
r^{-\frac{2\gamma}{p-1}}&\quad\text{as}~r\to0,\\
\epsilon r^{-(n-2\gamma)}&\quad\text{as}~r\to\infty.
\end{cases}
\end{equation*}
\end{thm}
For simplicity, denote by $u_*$ 
this fast decaying solution for $\epsilon=1$.
More precise asymptotics will be given in Propositions \ref{prop:decay-barw} and \ref{prop:unstable+}.

We remark that this  radially symmetric fast decaying solution can be used as the building block to construct the approximate solution in a gluing procedure. See \cite{Chan-DelaTorre} for a construction of a slow-growing solution (which serves the same purpose) at the threshold exponent $p=\frac{n}{n-2\gamma}$ for the case $\gamma=1$.

\section{Hardy type operators with fractional Laplacian}\label{section:Hardy}

Fix a constant $\kappa\in\mathbb R$.  Here we give a formula for the Green's function for the Hardy type operator in $\mathbb R^n$
\begin{equation*}\label{Hardy-operator}
 L_\kappa u:=(-\Delta)^{\gamma}u-\frac{\kappa}{r^{2\gamma}}u.
\end{equation*}
In the light of Section \ref{section:conformal}, it is useful to use conformal geometry to rewrite the fractional Laplacian on $\mathbb R^n$ in terms of the conformal fractional Laplacian $P_\gamma$ on the cylinder $M=\mathbb R\times\mathbb S^{n-1}$. Indeed,
from the conformal property \eqref{conformal-property}, setting
$$u=r^{-\frac{n-2\gamma}{2}}w, \quad w=w(t,\theta) \quad\text{for}\quad r={e^{-t}},$$ we have
\begin{equation*}\label{intertwining}
r^{\frac{n+2\gamma}{2}}L_\kappa u=P_\gamma w-\kappa w=:\mathcal Lw.
\end{equation*}
Our aim is to study invertibility properties for the equation
\begin{equation}\label{equation20}
\mathcal L w=h,\quad w=w(t,\theta),\quad t\in\mathbb R,\theta\in\mathbb S^{n-1}.
\end{equation}

Now consider the  projection of equation \eqref{equation20} over spherical harmonics: if we decompose
\begin{equation*}
w(t,\theta)=\sum_{m=0} w_m(t)E_m(\theta),\quad h(t,\theta)=\sum_{m=0} h_m(t)E_m(\theta),
\end{equation*}
then for $m=0,1\ldots$, $w_m=w_m(t)$ is a solution to
\begin{equation}\label{equation-introduction-m}
 {\mathcal L}_m w_m:=P_\gamma^{(m)} w_m-\kappa w_m=h_m\quad\text{on }\mathbb R.
\end{equation}
Recall Proposition \ref{prop:symbol} (taking into account that $Q_0=0$). Then, in Fourier variables, equation \eqref{equation-introduction-m} simply becomes
\begin{equation*}
(\Theta_\gamma^{(m)}(\xi)-\kappa)\hat w_m=\hat h_m,
\end{equation*}
where
\begin{equation}\label{symbol-isolated}
{\Theta}^{(m)}_{\gamma}(\xi)=2^{2\gamma}\frac{\Gamma \big(A_m
+\frac{1}{2}\xi i\big)\Gamma \big(A_m
-\frac{1}{2}\xi i\big)}
{\Gamma \big(B_m
+\frac{1}{2}\xi i\big)\Gamma \big(B_m
-\frac{1}{2}\xi i\big)}.
\end{equation}
The behavior of the equation depends on the zeroes of the symbol $\Theta^{(m)}_\gamma(\xi)-\kappa$. In any case, we can formally write
\begin{equation*}\label{w_m:Fourier}
w_m(t)=\int_{\mathbb R} \frac{1}{\Theta_\gamma^{(m)}(\xi)-\kappa} \hat h_m(\xi) e^{i\xi t}\,d\xi=\int_{\mathbb R}  \mathcal G_m(t-t')h_m(t')\,dt',
\end{equation*}
where the Green's function for the problem is given by
\begin{equation*}
\mathcal G_m(t)=\int_{\mathbb R} e^{i\xi t} \frac{1}{\Theta_\gamma^{(m)}(\xi)-\kappa}\,d\xi.
\end{equation*}
This statement is made rigorous in \cite{Ao-Chan-DelaTorre-Fontelos-Gonzalez-Wei} (see Theorem \ref{thm:Hardy-potential} below). First, observe that the symbol \eqref{symbol-isolated}
 can be extended meromorphically to the complex plane; this extension will be denoted simply by
\begin{equation*}
\Theta_m(z)
:=2^{2\gamma}\frac{\Gamma\big(A_m+\frac{1}{2}zi\big)\Gamma\big(A_m-\frac{1}{2}zi\big)}
{\Gamma\big(B_m+\frac{1}{2}zi\big)\Gamma\big(B_m-\frac{1}{2}zi\big)}, \quad z\in \mathbb C.
\end{equation*}

\begin{remark}
It is interesting to observe that $\Theta_m(z)=\Theta_m(-z),$
and that,  for $\xi\in\mathbb R$,
\begin{equation*}\label{symbol-limit}
\Theta_m(\xi)\asymp |m+\xi i|^{2\gamma},\quad \text{as}\quad |\xi|\to\infty,
\end{equation*}
and this limit is uniform in $m$. This also shows that, for fixed $m$, the behavior at infinity is the same as the one for the standard fractional Laplacian $(-\Delta)^\gamma$.
\end{remark}

There are several settings depending on the value of $\kappa$. Let us start with the stable case.

\begin{thm}[\cite{Ao-Chan-DelaTorre-Fontelos-Gonzalez-Wei}]\label{thm:poles}
Let $0\leq \kappa<\Lambda_{n,\gamma}$ and fix a non-negative integer $m$. 
Then the function $$\frac{1}{\Theta_m(z)-\kappa}$$ is meromorphic in $z\in\mathbb C$. More precisely,
its poles are located at points of the form $\tau_j\pm i\sigma_j$ and $-\tau_j\pm i\sigma_j$, where $\sigma_j>\sigma_0>0$ for $j=1,\ldots$, and $\tau_j\geq0$ for $j=0,1,\ldots$. In addition, $\tau_0=0$, and $\tau_j=0$ for all $j$ large enough. For such $j$, $\{\sigma_j\}$ is
 a strictly increasing sequence with no accumulation points.
\end{thm}

Now we go back to problem \eqref{equation-introduction-m}. From Proposition \ref{thm:poles} one immediately has:

\begin{cor}\label{remark:homogeneous}
For any fixed $m$, all solutions of the homogeneous problem ${\mathcal L}_m w=0$ are of the form
\begin{equation*}
\begin{split}
w_{h}(t)&=C_0^-e^{-\sigma_0 t}+C_0^+e^{\sigma_0 t}+\sum_{j=1}^\infty  e^{-\sigma_j t}
[C_j^-\cos({\tau_j t})+C_j'^-\sin({\tau_j t})]\\
&+\sum_{j=1}^\infty  e^{+\sigma_j t}
[C_j^+\cos({\tau_j t})+C_j'^+\sin({\tau_j t})]
\end{split}
\end{equation*}
for some real constants $C_j^-,C_j^+,C_j'^-,C_j'^+$, $j=0,1,\ldots$.
\end{cor}

In the next section, we will give a variation of constants formula to construct a particular solution to \eqref{equation-introduction-m}. As in the usual ODE case, one uses the solutions of the homogeneous problem as building blocks.

\subsection{The variation of constants formula}\label{subsection:variation}

Before we state our main theorem, let us recall a small technical lemma:

\begin{lemma}[\cite{Ao-Chan-DelaTorre-Fontelos-Gonzalez-Wei}]\label{lem:convolution}
Suppose
\[
f_1(t)=O(e^{-a|t|})
	\text{ as } |t|\to\infty
\quad \text{ and } \quad
f_2(t)=\begin{cases}
O(e^{-a_+ t})
	& \text{ as } t\to+\infty\\
O(e^{-a_-|t|})
	& \text{ as } t\to-\infty.
\end{cases}
\]
for $a>0$, $a+a_+>0$, $a+a_->0$, $a\neq a_+$, $a\neq a_-$.\footnote{When $a=a_+$, the upper bound is worsened to $O(te^{-at})$. A similar bound holds when $a=a_-$.} Then
\[
f_1\ast f_2(t)=\begin{cases}
O(e^{-\min\{a,a_+\} t})
	& \text{ as } t\to+\infty\\
O(e^{-\min\{a,a_-\}|t|})
	& \text{ as } t\to-\infty.
\end{cases}
\]
\end{lemma}

 From now on, once $m=0,1,\ldots$ has been fixed, we will drop the subindex $m$ in the notation if there is no risk of confusion. Thus, given $h=h(t)$, we consider the problem
\begin{equation}\label{problem-m}
\mathcal L_m w=h, \quad w=w(t).
\end{equation}

The variation of constants formula is one of the main results  in \cite{Ao-Chan-DelaTorre-Fontelos-Gonzalez-Wei}. Our version here is a minor restatement of the original result, to account for clarity. In addition, the proof has been simplified, so we give the complete arguments for statement \emph{b.} below.

\begin{thm}[\cite{Ao-Chan-DelaTorre-Fontelos-Gonzalez-Wei}]\label{thm:Hardy-potential}
Let $0\leq \kappa<\Lambda_{n,\gamma}$ and fix a non-negative integer $m$. 
Assume that the right hand side $h$ in \eqref{problem-m} satisfies
\begin{equation}\label{decay-h}
h(t)=\begin{cases}
O(e^{-\delta t}) &\text{as } t\to+\infty,\\
O(e^{\delta_0 t}) &\text{as } t\to-\infty,
\end{cases}
\end{equation}
for some real constants $\delta,\delta_0>-\sigma_0$. It holds:
\begin{itemize}
\item[\emph{a.}] A particular solution of \eqref{equation-introduction-m} can be written as
\begin{equation}\label{Green0}
w_p(t)=\int_{\mathbb R}  {\mathcal G}_m(t-t')h(t')\,dt',
\end{equation}
where
\begin{equation*}\label{Green1}
{\mathcal G}_m(t)=c_0e^{-\sigma_0 |t|}+\sum_{j=1}^\infty  e^{-\sigma_j |t|} [c_j\cos(\tau_j |t|)
+c'_j\sin(\tau_j |t|)],
\end{equation*}
for some precise real constants $c_j,c'_j$ depending on $\kappa,n,\gamma$.  Moreover, ${\mathcal G}_m$ is an even $\mathcal C^{\infty}$ function when $t\neq 0$.

\item[\emph{b.}] Suppose $\delta>\sigma_J$ for some $J=0,1,\dots$. Then the particular solution given by \eqref{Green0} satisfies
\begin{equation}\label{eq:wp-new}\begin{split}
w_p(t)
&=
	c_0 C_0 e^{-\sigma_0 t}
	+\sum_{j=1}^{J}
		c_j
		\left(
			C_j^1 \cos(\tau_j t)
			+C_j^2 \sin(\tau_j t)
		\right)
		e^{-\sigma_j t}
	+O(e^{-\min\{\delta,\sigma_{J+1}\}t})
\end{split}\end{equation}
as $t\to+\infty$, and
\[
w_p(t)=O(e^{-\min\{\delta_0,\sigma_0\}|t|})
\]
as $t\to-\infty$, where
\begin{equation}\label{constantD0}
C_0=
	\int_{\mathbb{R}}
		e^{\sigma_0 t'}
		h(t')
	\,dt',
\end{equation}
and for each $j=1,2,\dots$,
\[
C_j^1=
	\int_{\mathbb{R}}
		e^{\sigma_j t'}
		\cos(\tau_j t')
		h(t')
	\,dt',
\quad \text{ and } \quad
C_j^2=
	\int_{\mathbb{R}}
		e^{\sigma_j t'}
		\sin(\tau_j t')
		h(t')
	\,dt'.
\]
\end{itemize}
\end{thm}

\begin{proof}
The formula for the Green's function $\mathcal{G}_m$ in \emph{a.} follows directly from \cite{Ao-Chan-DelaTorre-Fontelos-Gonzalez-Wei}. We will prove \emph{b.}, building on the first part. Without loss of generality we fix $m=0$ and suppress the subscript.

Note that $\mathcal{G}=O(e^{-\sigma_0|t|})$ as $t\to\pm\infty$. Then,  since $\delta,\delta_0>-\sigma_0$, Lemma \ref{lem:convolution} implies that,  as $t\to-\infty$,
\[
w_p(t)=O(e^{-\min\{\delta_0,\sigma_0\}|t|}).
\]

As for $t\to+\infty$, let us write
\[
\mathcal{G}(t)
=
	c_0 e^{-\sigma_0|t|}
	+\sum_{j=1}^{J}  e^{-\sigma_j |t|} [c_j\cos(\tau_j |t|)+c'_j\sin(\tau_j |t|)]
	+\mathcal{G}^J(t),
\]
where
\[
\mathcal{G}^J(t)
=
	\sum_{j=J+1}^{\infty}  e^{-\sigma_j |t|} [c_j\cos(\tau_j |t|)+c'_j\sin(\tau_j |t|)].
\]
The particular solution in \eqref{Green0} is then given by
\[
w_p
=
	W_0
	+\sum_{j=1}^{J}[W_j+W_j']
	+\mathcal{G}^J \ast h,
\]
for
\[\begin{split}
W_0(t)
&=
	c_0\int_{\mathbb{R}}
		e^{-\sigma_0|t-t'|}
		h(t')
	\,dt',\\
W_j(t)+W_j'(t)
&=
	\int_{\mathbb{R}}
		e^{-\sigma_j|t-t'|}
		[c_j\cos(\tau_j|t-t'|)+c'_j\sin(\tau_j|t-t'|)]
		h(t')
	\,dt'.
\end{split}\]
By Lemma \ref{lem:convolution}, using the facts that $\mathcal{G}^J=O(e^{-\sigma_{J+1}|t|})$ as $t\to\pm\infty$, $\delta>\sigma_J>-\sigma_{J+1}$ and $\delta_0>-\sigma_0>-\sigma_{J+1}$, we have
\begin{equation}\label{eq:wp-est-1}
\mathcal{G}^J \ast h(t)
=O(e^{-\min\{\delta,\sigma_{J+1}\}|t|})
	\quad \text{ as } t\to+\infty.
\end{equation}
Next we turn to the terms $W_0$ and $W_j$, $j=1,\dots,J$. Their estimates are the same in spirit but that for $W_0$ is simpler since $\tau_0=0$. Using that $\delta>\sigma_J>\sigma_0$ and $\delta_0>-\sigma_0$, we have $e^{\sigma_0 \cdot} h \in L^1(\mathbb{R})$ and we can write
\begin{equation}\label{eq:wp-est-2}\begin{split}
W_0(t)
&=
	c_0\int_{-\infty}^{t}
		e^{-\sigma_0 t}
		e^{\sigma_0 t'}
		h(t')
	\,dt'
	+c_0\int_{t}^{+\infty}
		e^{\sigma_0 t}
		e^{-\sigma_0 t'}
		h(t')
	\,dt'\\
&=
	c_0 e^{-\sigma_0 t}
	\left\{
		\int_{-\infty}^{+\infty}
		-\int_{t}^{+\infty}
	\right\}
		e^{\sigma_0 t'}
		h(t')
	\,dt'
	+c_0 e^{\sigma_0 t}
	\int_{t}^{+\infty}
		e^{-\sigma_0 t'}
		h(t')
	\,dt'\\
&=
	c_0 C_0 e^{-\sigma_0 t}
	+O\left(
		e^{-\sigma_0 t}
		\int_{t}^{+\infty}
			e^{\sigma_0 t'}
			e^{-\delta t'}
		\,dt'
	\right)
	+O\left(
		e^{\sigma_0 t}
		\int_{t}^{+\infty}
			e^{-\sigma_0 t'}
			e^{-\delta t'}
		\,dt'
	\right)\\
&=
	c_0 C_0 e^{-\sigma_0 t}
	+O(e^{-\delta t}),
\end{split}\end{equation}
as $t\to+\infty$. Similarly, for all $j=1,\dots,J$, we have $e^{\sigma_j \cdot} h \in L^1(\mathbb{R})$ in view of $\delta>\sigma_J \geq \sigma_j$ and $\delta_0>-\sigma_0$. We can therefore compute
\begin{equation}\label{eq:wp-est-3}\begin{split}
W_j(t)
&=
	c_j\int_{-\infty}^{t}
		e^{-\sigma_j t}
		e^{\sigma_j t'}
		\left(
			\cos(\tau_j t) \cos(\tau_j t')
			+\sin(\tau_j t) \sin(\tau_j t')
		\right)
		h(t')
	\,dt'\\
&\quad\;
	+c_j\int_{t}^{+\infty}
		e^{\sigma_j t}
		e^{-\sigma_j t'}
		\cos(\tau_j(t'-t))
		h(t')
	\,dt'\\
&=
	c_j e^{-\sigma_j t} \cos(\tau_j t)
	\left(
		\int_{-\infty}^{+\infty}
		-\int_{t}^{+\infty}
	\right)
		e^{\sigma_j t'}
		\cos(\tau_j t')
		h(t')
	\,dt'\\
&\quad\;
	+c_j e^{-\sigma_j t} \sin(\tau_j t)
	\left(
		\int_{-\infty}^{+\infty}
		-\int_{t}^{+\infty}
	\right)
		e^{\sigma_j t'}
		\sin(\tau_j t')
		h(t')
	\,dt'\\
&\quad\;
	+c_j e^{\sigma_j t}
	\int_{t}^{+\infty}	
		e^{-\sigma_j t'}
		\cos(\tau_j(t'-t))
		h(t')
	\,dt'\\
&=
	c_j C_j^1 e^{-\sigma_j t} \cos(\tau_j t)
	+O\left(
		e^{-\sigma_j t}
		\int_{t}^{+\infty}
			e^{\sigma_j t'}
			e^{-\delta t'}
		\,dt'
	\right)\\
&\quad\;
	+c_j C_j^2 e^{-\sigma_j t} \sin(\tau_j t)
	+O\left(
		e^{-\sigma_j t}
		\int_{t}^{+\infty}
			e^{\sigma_j t'}
			e^{-\delta t'}
		\,dt'
	\right)\\
&\quad\;
	+O\left(
		e^{\sigma_j t}
		\int_{t}^{+\infty}
			e^{-\sigma_j t'}
			e^{-\delta t'}
		\,dt'
	\right)\\
&=
	c_j
	\left(
		C_j^1 \cos(\tau_j t)
		+C_j^2 \sin(\tau_j t)
	\right)
	e^{-\sigma_j t}
	+O(e^{-\delta t}),
\end{split}\end{equation}
as $t\to+\infty$, and similarly for $W_j'$.
 Putting together \eqref{eq:wp-est-1}, \eqref{eq:wp-est-2} and \eqref{eq:wp-est-3} yields \eqref{eq:wp-new}.
\end{proof}

\smallskip

We also look at the case when $\kappa$ leaves the stability regime. In order to simplify the presentation, we only consider the projection $m=0$ and the equation
\begin{equation}\label{equation-radial}
{\mathcal L}_0 w=h,\quad w=w(t).
\end{equation}
Moreover, we assume that only the first pole leaves the stability regime, which happens if $\Lambda_{n,\gamma}<\kappa<\Lambda'_{n,\gamma}$  for some constant $\Lambda'_{n,\gamma}>\Lambda_{n,\gamma}$. Then, in contrast to Theorem \ref{thm:poles}, we will have two additional real poles $\tau_0$ and $-\tau_0$.
\begin{prop}\label{prop:unstable}
Let $\Lambda_{n,\gamma}<\kappa<\Lambda'_{n,\gamma}$. 
Assume the decay condition \eqref{decay-h} for $h$ as in Theorem \ref{thm:Hardy-potential}.
It holds:
\begin{itemize}
\item[\textit{i.}]
If $\delta,\delta_0> 0$, then
a particular solution of \eqref{equation-radial} can be written as
\begin{equation*}\label{v0}
w_0(t)=\int_{\mathbb R}  {\mathcal G_0}(t-t')h(t')\,dt',
\end{equation*}
where
\begin{equation*}
\mathcal G_0(t)=
c_0\sin({\tau_0t})\chi_{(-\infty,0)}(t)
+\sum_{j=1}^\infty e^{-\sigma_j |t|} [c_j\cos(\tau_j t)+c_j'\sin(\tau_j |t|)]
\end{equation*}
for some constants $c_j,c_j'$, $j=0,1,\ldots$. Moreover, $\mathcal G_0$ is a $\mathcal C^{\infty}$ function when $t\neq 0$.

\item[\textit{ii.}] The analogous statements to  Theorem \ref{thm:Hardy-potential}, {\em b.},  and Corollary \ref{remark:homogeneous} hold.
\end{itemize}
\end{prop}



As we have mentioned,
Theorem \ref{thm:Hardy-potential} can be interpreted in terms of the variation of constants method. This in turn allows the reformulation of the non-local problem \eqref{problem-m} into an infinite system of second order ODE's. Since the theory is particularly nice when all the $\tau_j$ are zero, we present it separately from the general case, in which complex notations are used.


\begin{cor}\label{cor:ODEsystem}
Take $w$ as in \eqref{Green0} from Theorem \ref{thm:Hardy-potential}. In the special case that $\tau_j=0$ for all $j$, then
\begin{equation}\label{w-sum}
w(t)=\sum_{j=0}^\infty c_j w_j(t),
\end{equation}
where
\begin{equation*}
w_j(t):=\int_{\mathbb R}e^{-\sigma_j |t-t'|}h(t')\,dt'.
\end{equation*}
Moreover, $w_j$ is a particular solution to the second order ODE
 \begin{equation}\label{ODE-j}
w_j''(t)-\sigma_j^2w_j(t)=-2\sigma_j h(t).
\end{equation}
\end{cor}

In general, such $w$ can be written as a real part of a series whose terms solve a complex-valued second order ODE.

\begin{cor}\label{cor:ODEsystem-complex}
Given $w$ as in \eqref{Green0}, we define the complex-valued functions $w_j:\mathbb{R}\to\mathbb{C}$ by
\[
w_j=e^{-(\sigma_j+i\tau_j)|\cdot|}
    \ast h.
\]
They satisfy the second order ODE
\begin{equation*}\label{eq:w_j-complex}
-\dfrac{1}{\sigma_j+i\tau_j}w_j''
+(\sigma_j+i\tau_j)w_j
=2h,
\end{equation*}
and the original (real-valued) function $w$ can be still recovered by 
\begin{equation*}\label{eq:w-sum-complex}
w(t)
=\mathrm{Re\,}\sum_{j=0}^{\infty}
	c_jw_j(t).
\end{equation*}
\end{cor}

Another interesting fact is that, for $\kappa=0$, equation \eqref{Green0} is simply the expansion of the Riesz potential for the fractional Laplacian. Indeed, let us recall the following  (this is a classical formula; see, for instance, \cite{Calvez-Carrillo-Hoffmann} and the references therein):

\begin{prop}\label{prop:Riesz}
Assume that $u$ is the Riesz potential of a compactly supported radial density $\tilde h=\tilde h(r)$ in $\mathbb R^n$. It is always possible to write (up to multiplicative constant)
\begin{equation}\label{Riesz}
u=|x|^{2\gamma-n}\ast \tilde h=r^{2\gamma-n}\int_{\eta=0}^{r}\vartheta_{n,\gamma}
\left(\frac{\eta}{r}\right)\tilde h(\eta)\eta^{n-1}d\eta+
\int_{\eta=r}^{\infty}\eta^{2\gamma-n}\vartheta_{n,\gamma}\left(\frac{r}{\eta}\right)\tilde h(\eta)\eta^{n-1}d\eta,
\end{equation}
where
\[
\vartheta_{n,\gamma}(z)=
d_{n}\cdot\Hyperg\left(-\gamma+\tfrac{n}{2},1-\gamma;\tfrac{n}{2};z^{2}\right).
\]
\end{prop}

In order to relate to our setting, first we need to shift the information from $r=\infty$ to the origin, so we set $t=\log r$ (note the sign change with respect to the above!). If we denote $u=r^{-\frac{n-2\gamma}{2}}w$, $\tilde h=r^{-\frac{n+2\gamma}{2}}h$, and take the Taylor expansion of the Hypergeometric function in \eqref{Riesz},  we obtain an expansion of the type given in \eqref{w-sum}, since in this case we have
\begin{equation*}\label{inditial:zero}
\sigma_j=\dfrac{n-2\gamma}{2}+2j, \quad \tau_j=0,\quad j=0,1,\ldots
\end{equation*}
Thus we can interpret  \eqref{w-sum} as the generalization of \eqref{Riesz} in the presence of a potential term with $\kappa\neq 0$.

\subsection{Frobenius theorem}\label{subsection:Frobenius}

 In the following, we will concentrate just on radial solutions (which correspond to the $m=0$ projection above), but the same arguments would work for any $m$. In particular, we study the kernel of the fractional Laplacian operator with a radially symmetric Hardy-type potential, which is given by the non-local ODE
\begin{equation}\label{equation:kernel-phi}
L \phi=(-\Delta)^\gamma \phi-\frac{\mathcal V(r)}{r^{2\gamma}}\phi=0,\quad \phi=\phi(r),
\end{equation}
We have shown that this equation is equivalent to
\begin{equation}\label{equation:kernel}
\mathcal L_0 w=P_\gamma^{(0)}w-\mathcal V(t)w=0,\quad w=w(t),
\end{equation}
where we have denoted $\phi=r^{-\frac{n-2\gamma}{2}}w$, $r=e^{-t}$.

The arguments here are based on an iteration scheme from  \cite{Ao-Chan-DelaTorre-Fontelos-Gonzalez-Wei} (Sections 6 and 7). However, the restatement of Theorem \ref{thm:Hardy-potential} that we have presented here makes the proofs more transparent, so we give here full details for convenience of the reader.

 Assume, for simplicity, that we are in the stable case, this is, in the setting of Theorem \ref{thm:Hardy-potential}. In the unstable case, we have similar results by applying Proposition \ref{prop:unstable}.

We fix any radially symmetric, smooth potential $\mathcal V(t)$ with the asymptotic behavior
\begin{equation}\label{potential1}
\mathcal V(t)=
\begin{cases}
\kappa+O(e^{-qt}),
 & \text{if } t\to+\infty, \\
O(e^{q_1 t}),
 & \text{if } t\to -\infty,
\end{cases}
\end{equation}
 for some $q,q_1>0$, and such that  $0\leq \kappa<\Lambda_{n,\gamma}$.

The indicial roots for problem \eqref{equation:kernel} as $t\to +\infty$ are calculated by looking at the limit problem
\begin{equation*}
P_\gamma^{(0)}w-\kappa w=0,\quad w=w(t),
\end{equation*}
which are given in Theorem \ref{thm:poles}. Indeed, these are of the form
 \begin{equation*}
 \{\sigma_j \pm i\tau_j\}, \, \{-\sigma_j \pm i\tau_j\},\quad \tau_0=0,\,\tau_j=0 \text{ for }j \text{ large enough}.
 \end{equation*}

The results from the previous section imply that the behavior of solutions to \eqref{equation:kernel} are governed by the indicial roots of the problem. This is a Frobenius type theorem for a non-local equation.

\begin{prop}[\cite{Ao-Chan-DelaTorre-Fontelos-Gonzalez-Wei}]\label{prop:behavior-indicial}
Fix any potential $\mathcal V(t)$ as above. Let  $w=w(t)$ be any solution to \eqref{equation:kernel}
satisfying that
$w=O(e^{-\alpha_0 |t|})$ as $|t|\to\infty$ for some $\alpha_0>-\sigma_0$. Then there exists a non-negative integer $j$ such that either
\[
w(t)=(a_j+o(1)) e^{-\sigma_j t}
	\quad \text{ as }
t\to+\infty,
\]
for some real number $a_j\neq 0$, or
\[
w(t)=
	\left(
		a_j^1 \cos(\tau_j t)
		+ a_j^2 \sin(\tau_j t)
		+ o(1)
	\right)
	e^{-\sigma_j t},
\]
for some real numbers $a_j^1, a_j^2$ not vanishing simultaneously.
\end{prop}

We remark that a similar conclusion holds at $-\infty$.

\begin{proof}
Write the equation satisfied by $w$ \eqref{equation:kernel} as
\begin{equation*}
P_\gamma^{(0)} w-\kappa w=(\mathcal V-\kappa)w=:h.
\end{equation*}
Note that, by \eqref{potential1},
\[\begin{split}
w(t)=\begin{cases}
O(e^{-\alpha_0 t})
	& \text{ as } t\to+\infty,\\
O(e^{-\alpha_0 |t|})
	& \text{ as } t\to-\infty,
\end{cases}
\qquad
\mathcal{V}(t)-\kappa=\begin{cases}
O(e^{-qt})
	& \text{ as } t\to+\infty,\\
O(1)
	& \text{ as } t\to-\infty.
\end{cases}
\end{split}\]
We follow closely the proof of Theorem \ref{thm:Hardy-potential} but, this time, since the right hand side $h$ depends on the solution $w$, we need to take that into account in the iteration scheme. By part {\emph a.}, a particular solution $w_p=\mathcal{G}\ast h$ satisfies
\[
w_p(t)=\begin{cases}
O(e^{-\min\{\alpha_0+q,\sigma_0\} t})
	& \text{ as } t\to+\infty,\\
O(e^{-\alpha_0 |t|})
	& \text{ as } t\to-\infty.
\end{cases}
\]
This is in fact the behavior of $w$, since the addition of any kernel element would create an exponential growth of order at least $\sigma_0$ which is not permitted by assumption. As a consequence, we obtain a better decay of $w$ as $t\to+\infty$, so we can iterate this argument to arrive at
\[
w(t)=\begin{cases}
O(e^{-\sigma_0 t})
	& \text{ as } t\to+\infty,\\
O(e^{-\alpha_0 |t|})
	& \text{ as } t\to-\infty.
\end{cases}
\]
Using Theorem \ref{thm:Hardy-potential} \emph{b.} with $J=0$, we have (in its notation)
\[
w(t)=\begin{cases}
	\displaystyle
	c_0 C_0
	e^{-\sigma_0 t}
	+O(e^{-\min\{\sigma_0+q,\sigma_1\} t})
		& \text{ as } t\to+\infty,\\
	O(e^{-\alpha_0|t|})
		& \text{ as } t\to-\infty.
\end{cases}
\]
When 
$C_0\neq0$, we have that
$$w(t)=(a_0+o(1)) e^{-\sigma_0 t}\quad
	\text{ as }\quad t\to +\infty,$$
for a non-zero constant $a_0=c_0 C_0$.

Otherwise, in the case $C_0=0$, we iterate this process to yield
\[
w(t)=\begin{cases}
O(e^{-\sigma_1 t})
	& \text{ as } t\to+\infty,\\
O(e^{-\alpha_0|t|})
	& \text{ as } t\to-\infty.
\end{cases}
\]
An application of Theorem \ref{thm:Hardy-potential} \emph{b.} with $J=1$ yields
\[
w(t)=\begin{cases}
c_1\left(
    C_1^1 \cos(\tau_1 t)
    +C_1^2 \sin(\tau_1 t)
\right)e^{-\sigma_1 t}
+O(e^{-\min\{\sigma_1+q,\sigma_2\}t})
    & \text{ as } t\to+\infty,\\
O(e^{-\alpha_0|t|})
    & \text{ as } t\to-\infty,
\end{cases}
\]
which has the exact (oscillating if $\tau_1>0$) behavior of order $e^{-\sigma_1 t}$ as $t\to+\infty$ unless both coefficients $C_1^1,C_1^2$ vanish, in which case the decay of $w$ is further improved through an iteration, i.e.
\[
w(t)=\begin{cases}
O(e^{-\sigma_2 t})
	& \text{ as } t\to+\infty,\\
O(e^{-\alpha_0|t|})
	& \text{ as } t\to-\infty.
\end{cases}
\]

We use induction on $J$, the number of additional isolated terms in $\mathcal{G}$. Depending on the vanishing properties of the coefficients $D_J^1,D_J^2$, this gives either an exact (signed or oscillating) behavior of $w$ as $e^{-\sigma_J t}$, or $w\equiv 0$ by unique continuation when no such $J$ exists. In the stable case, unique continuation was proved in \cite{Fall-Felli} using a monotonicity formula, while in the unstable case it follows from \cite{Ruland}, where Carleman estimates were the crucial ingredient.
\end{proof}

\section{Non-degeneracy}\label{section:non-degeneracy}

For the critical case $p=\frac{n+2\gamma}{n-2\gamma}$, non-degeneracy for the standard bubble $w_{\infty}$ (given in Proposition \ref{bubble}) has been considered separately by \cite{Frank:uniqueness,Chen-Frank-Weth} and \cite{dps}.

We thus restrict to the subcritical case $$p\in\left(\frac{n}{n-2\gamma},\frac{n+2\gamma}{n-2\gamma}\right).$$
Let $u_*$ be the model solution constructed in Theorem \ref{existence}, and set $L_*$ be the linearized operator for \eqref{problem} around this particular solution:
\begin{equation*}\label{eq:linearized-phi}
L_*\varphi=(-\Delta)^\gamma \varphi -\frac{\mathcal V_*}{r^{2\gamma}}\varphi,
\end{equation*}
for the potential $\mathcal V_*=r^{2\gamma}pA(u_*)^{p-1}$. Recall that $\mathcal V_*$ converges to the constant $\kappa:=pA$ as $r\to 0$.

In terms of the $t$ variable we can write as follows: set $r=e^{-t}$,
\begin{equation*}
\phi=r^{\frac{n-2\gamma}{2}}w,\quad u_*=r^{-\frac{2\gamma}{p-1}}v_*.
\end{equation*}
The linearized operator is now
\begin{equation*}
\mathcal L_* w:= r^{\frac{n+2\gamma}{2}}L_* \phi=P_\gamma^{(0)} w-\mathcal V_*(t)w,
\end{equation*}
for the potential
\begin{equation}\label{potential}
\mathcal V_*(t)=pA(v_*)^{p-1}=
\begin{cases}
\kappa+O(e^{-qt}),
 &\text{if } t\to+\infty, \\
O(e^{q_1 t}),
 & \text{if } t\to -\infty,
\end{cases}
\end{equation}
for some $q,q_1>0$.
\medskip

Now we look for radially symmetric solutions to
\begin{equation}\label{problem20}L_* \phi=0,\quad \phi=\phi(r)\end{equation}
which, in terms of the $t$ variable, is equivalent to
\begin{equation*}
P_\gamma^{(0)} w-\mathcal V_*(t)w=0, \quad w=w(t).
\end{equation*}
Recall that the potential $\mathcal V_*$ is given in \eqref{potential} and it is of the type considered in Proposition \ref{prop:behavior-indicial}, so we know that the asymptotic expansion of solutions are governed by the indicial roots of the problem. These are given by:

\begin{lemma}[\cite{Ao-Chan-DelaTorre-Fontelos-Gonzalez-Wei}]
\label{lem:num}
Consider the equation \eqref{equation:kernel} for the potential $\mathcal V_*$ as in \eqref{potential}. Then:
\begin{itemize}
\item As $t\to+\infty$, the indicial roots for  the problem $P_\gamma^{(0)}w-\kappa w=0$ are given by  sequences $\{\sigma_j\pm i\tau_j\}_{j=0}^\infty$, $\{-\sigma_j\pm i\tau_j\}_{j=0}^\infty$. Moreover, there exists $p_0$ such that for $p<p_0$, we are in the setting of Theorem \ref{thm:Hardy-potential} (stable case), while for $p>p_0$, we are in the setting of Proposition \ref{prop:unstable} (unstable case).
\item As $t\to-\infty$, the indicial roots for the $P_\gamma^{(0)}w=0$ are given by  sequences
      $\{\pm\varsigma_j\}_{j=0}^\infty$, where $\varsigma_j=\frac{n-2\gamma}{2}+j$.
\end{itemize}
\end{lemma}

Assume again, without loss of generality, that we are in the stable case and that all the $\tau_j=0$, $j=0,1,\ldots$. \medskip

We know from Theorem \ref{existence} that  $u_*=(1+o(1))r^{-\frac{2\gamma}{p-1}}$ as $r\to 0$. Let us find the next term in the expansion, and show that it is given by the first indicial root. For this, set
$$\overline u=u_*-r^{-\frac{2\gamma}{p-1}},$$
which is a solution of
\begin{equation}\label{equation30}
(-\Delta)^\gamma \overline u=A\left[(u_*)^p-r^{-p\frac{2\gamma}{p-1}}\right].
\end{equation}
Instead of the fractional Laplacian $(-\Delta)^\gamma$ we prefer to use the shifted operator \eqref{conformal-property} and thus we set
\begin{equation}\label{eq:w-bar}
\overline w=r^{\frac{n-2\gamma}{2}}\overline u.
\end{equation}
Then:

\begin{prop}\label{prop:decay-barw} There exists $a>0$ such that
\begin{equation*}
\overline w(t)=(a+o(1)) e^{-\sigma_0 t}\quad\text{as}\quad t\to +\infty.
\end{equation*}
\end{prop}

\begin{proof}
Equation \eqref{equation30} implies, for  $\overline v=r^{\frac{2\gamma}{p-1}}\overline u$, that for $\kappa=pA$,
\begin{equation*}
(-\Delta)^\gamma \overline u-\frac{\kappa}{r^{2\gamma}}\overline u=Ar^{-p\frac{2\gamma}{p-1}}\left[(1+\overline v)^p-1-p\overline v\right]=:\overline h > 0,
\end{equation*}
since $p>1$, unless $\overline v\equiv 0$.
Now set $\overline w=r^{\frac{n-2\gamma}{2}} \overline u$, then the above equation is equivalent to
\begin{equation}\label{RHS}
P_\gamma^{(0)} \overline w-\kappa\overline w=r^{\frac{n+2\gamma}{2}} \overline h=:h>0.
\end{equation}
Now,  since $\overline{w}(t)$ decays both as $t\to\pm\infty$, we can use Proposition \ref{prop:behavior-indicial}, taking into account that $a_0$  (or equivalently, $C_0$ from \eqref{constantD0}) cannot vanish due to  \eqref{RHS}.
\end{proof}

Now we go back to problem \eqref{problem20}. First,
since \eqref{problem} is invariant under rescaling, it is well known that
\begin{equation*}
\phi_*:=r\partial_r u_*+\tfrac{2\gamma}{p-1}u_*
\end{equation*}
belongs to the kernel of $L_*$. We will show that this is the only possibility (\emph{non-degeneracy} of $u_*$).

Set $w_*$ defined as $w_*=r^{\frac{n-2\gamma}{2}}\phi_*$. It is a radially symmetric, smooth  solution to $\mathcal L_* w=0$ that decays both as $t\to \pm\infty$. In particular, from Proposition \ref{prop:decay-barw} one has
\begin{equation*}
w_*(t)=(a+o(1)) e^{-\sigma_0 t}\quad\text{as}\quad t\to +\infty
\end{equation*}
for some $a\neq 0$.

\subsection{Wro\'{n}skians in the non-local setting}

We will show that $u_*$ is non-degenerate, this is, the kernel of $L_*$ consists on multiples of $\phi_*$. For simplicity, we will concentrate just on radial solutions (which correspond to the $m=0$ projection above), but the same argument would work for any $m$. Also, let us restrict to the stable case in order not to have a cumbersome notation. The main idea is to write down a quantity that would play the role of Wro\'{n}skian for a standard ODE. We provide two approaches: first, using the extension (Lemma \ref{lemma:Wronskian1}) and then working directly in $\mathbb R^n$ (Lemma \ref{lemma:Wronskian2}).

\begin{lemma}\label{lemma:Wronskian1}
Let $w_i=w_i(t)$, $i=1,2$, be two radially symmetric solutions to $\mathcal L_*w=0$ and set $W_i$, $i=1,2$, the corresponding extensions from Proposition \ref{prop:extension}.
Then, the Hamiltonian quantity
\begin{equation*}\label{Hamiltonian-tilde}
\begin{split}
\tilde H_\gamma(t):=\int_{0}^{\rho_0}  {\rho}^{1-2\gamma} e_2(\rho)[W_1\partial_t W_2-W_2\partial_t W_1]\,d\rho\\
\end{split}
\end{equation*}
 is constant along trajectories.
\end{lemma}

\begin{proof}
The proof is the same as the one from Theorem \ref{thctthamiltonian}.
 however, we present it for completeness. By straightforward calculation, using \eqref{equation-extension} for the second equality and integrating by parts in the third equality,
\begin{equation*}
\begin{split}
\partial_t&\int_{0}^{\rho_0}  {\rho}^{1-2\gamma} e_2(\rho) [W_1\partial_t W_2-W_2\partial_t W_1]\,d\rho\\
&=\int_{0}^{\rho_0}  {\rho}^{1-2\gamma} e_2(\rho) [W_1\partial_{tt} W_2-W_2\partial_{tt} W_1]\,d\rho\\
&= \int_{0}^{\rho_0}  \big(-W_1\partial_\rho(e_1(\rho)\rho^{1-2\gamma} \partial_{\rho} W_2)-W_2\partial_\rho(e_1(\rho)\rho^{1-2\gamma} \partial_{\rho} W_1\big)\,d\rho\\
&= \lim_{\rho\to 0}e_1(\rho)\big(w_1\rho^{1-2\gamma} \partial_{\rho} w_2-w_2\rho^{1-2\gamma} \partial_{\rho} w_1\big)\\
&=\frac{-1}{d_\gamma^*} \big ( w_1P_\gamma^{(0)} w_2-w_2P_\gamma^{(0)} w_1\big)=0
\end{split}
\end{equation*}
since both $w_1$ and $w_2$ satisfy the same equation $\mathcal L_* w=0$.
\end{proof}

Now define the Wro\'{n}skian of two solutions for the ODE \eqref{ODE-j}
\begin{equation*}
\mathcal W_j[w,\tilde w]=w_j\tilde w_j'-w_j'\tilde w_j,
\end{equation*}
and its weighted sum in $j=0,1,\ldots$
\begin{equation}\label{Wronskian}
\mathcal W[w,\tilde w]=\sum_{j=0}^{\infty}
	\dfrac{c_j}{\sigma_j}\mathcal W_j[w,\tilde w],
\end{equation}
for the constants given in Theorem \ref{thm:Hardy-potential}.

\begin{lemma}\label{lemma:Wronskian2}
Let $w,\tilde w$ be two radially symmetric solutions of $
\mathcal L_*w=0$. Then the Wro\'{n}skian quantity from \eqref{Wronskian} satisfies
\[
\mathcal{W}[w,\tilde w]'=0.
\]
\end{lemma}

\begin{proof}
Just recall the ODE system formulation \eqref{w-sum} and \eqref{ODE-j}, consider the corresponding Wro\'{n}skian sequence and sum in $j$.
Since $w$ and $\tilde w$ both satisfy
\[
w_j''-\sigma_j^2 w_j = -2\sigma_j \mathcal{V}_*(t)w,
\]
 we see that
\[\begin{split}
\mathcal W_j[w,\tilde w]'
&=w_j\tilde w_j''-w_j''\tilde w_j
=-2\sigma_j \mathcal{V}_*(t)(w_j\tilde w-\tilde w_jw),
\end{split}\]
and the conclusion follows.
\end{proof}

\begin{prop}\label{prop:other-solutions}
Any other radially symmetric solution to $\mathcal L_* w=0$ that decays both at $\pm \infty$ must be a multiple of $w_*$.
\end{prop}

\begin{proof}
The proof is standard as is based on the previous Hamiltonian identities, applied to $w$ and $w_*$. Let us write the proof using Lemma \ref{lemma:Wronskian2}, for completeness. It holds
\[
\mathcal W[w,w_*](t)\equiv \lim_{t\to+\infty}\mathcal W[w,w_*](t)=0.
\]
We claim that $w$ and $w_*$ must have the same asymptotic expansion as $t\to+\infty$. By Proposition \ref{prop:behavior-indicial}, the leading order terms of $w$ and $w_*$ as $t\to+\infty$ must be an exponential with some indicial root, i.e.
\[
w(t)=a_{j_0}(1+o(1)) e^{-\sigma_{j_0} t},
	\quad
w_*(t)=a_{j_0^*}(1+o(1)) e^{-\sigma_{j_0^*} t},
\]
for some non-negative integers $j_0,j_0^*$ and non-zero $a_{j_0}, a_{j_0^*}$. Note that these expressions ``can be differentiated'' in the sense that
\[
w'(t)=a_{j_0}(-\sigma_{j_0}+o(1)) e^{-\sigma_{j_0} t},
	\quad
w_*'(t)=a_{j_0^*}(-\sigma_{j_0^*}+o(1)) e^{-\sigma_{j_0^*} t}.
\]
Since
\[\begin{split}
\mathcal W[w,w_*](t)
&=
	a_{j_0} a_{j_0^*}
	(\sigma_{j_0}-\sigma_{j_0^*}+o(1))
	e^{-(\sigma_{j_0}+\sigma_{j_0^*})t}
\quad \text{ as } t \to +\infty,
\end{split}\]
we obtain that $\sigma_{j_0}=\sigma_{j_0^*}$. From the bilinearity of $\mathcal W[w,w_*]$, we can assume by rescaling that $a_{j_0}=a_{j_0^*}=1$.

We now look at the next order. We suppose
\[
w(t)
=
	e^{-\sigma_{j_0} t}
	+a_\alpha(1+o(1))e^{-\alpha t},
		\quad
w_*(t)
=
	e^{-\sigma_{j_0^*} t}
	+a_{\alpha^*}(1+o(1))e^{-\alpha^* t},
\]
for some complex numbers $\alpha,\alpha^*$ with $\mathrm{Re\,} \alpha,\mathrm{Re\,}\alpha^*\geq j_0$ and non-zero real numbers $a_{\alpha}, a_{\alpha^*}$.
A direct computation of the Wro\'{n}skian yields
\[\begin{split}
\mathcal W[w,w_*](t)
&=
	a_{\alpha}(\alpha-\sigma_{j_0})
	(1+o(1))
	e^{-(\sigma_{j_0}+\alpha)t}\\
&\quad\;
	-a_{\alpha^*}(\alpha^*-\sigma_{j_0})
	(1+o(1))
	e^{-(\sigma_{j_0}+\alpha^*)t},
\quad \text{ as } t\to+\infty.
\end{split}\]
In order that $\mathcal W[w,w_*]\equiv0$, the next order exponents $\alpha$ and $\alpha^*$ must be matched and the same expression also tells us that $a_{\alpha}=a_{\alpha^*}$.

Inductively we obtain that, once $w$ and $w_*$ are rescaled to match the leading order, they have the same asymptotic expansion up to any order, as $t\to+\infty$. Unique continuation, as applied in the above results, yields the result.
\end{proof}

\subsection{The unstable case}

Let us explain the modifications that are needed in the above for the unstable case. We  consider only radially symmetric solutions (the $m=0$ mode).

First we recall some facts on the indicial roots as $t\to +\infty$. It holds that $\sigma_0=0$, $\tau_0\neq 0$. We also know that $\sigma_j>0$ for all 
$j\geq 1$,
and $\tau_j=0$ for all $j\geq J$, for some $J$ large enough. A more precise estimate for $J$ would be desirable, but it would be too technical.

\begin{prop}\label{prop:unstable+}
Let 
$\kappa$ be as in Proposition \ref{prop:unstable} and $\overline{w}$ be as defined by \eqref{eq:w-bar}. As $t\to +\infty$, either
\begin{itemize}
\item[\emph{i}.] $\overline w=a_0^1 \cos(\tau_0 t)+a_0^2 \sin(\tau_0 t)+o(1)$, where $a_0^1$, $a_0^2$ do not vanish simultaneously; or
\item[\emph{ii}.] for some
    positive integer $j_0\leq J-1$
$$\overline w =\left[ a_{j_0}^1 \cos(\tau_{j_0} t)+a_{j_0}^2 \sin(\tau_{j_0} t)+o(1)\right]e^{-\sigma_{j_0}t},$$
where $a_{j_0}^1$, $a_{j_0}^2$ do not vanish at the same time; or
\item[\emph{iii}.] There exists $a_J>0$ such that
$$\overline w =(a_J+o(1)) e^{-\sigma_J t}.$$
\end{itemize}
\end{prop}

\begin{proof}
We follow the ideas in Proposition \ref{prop:behavior-indicial}. Note that
if the integrals
\begin{equation*}
\int_{\mathbb R} \cos(\tau_0 t)h(t)\,dt,\quad \int_{\mathbb R} \sin(\tau_0 t)h(t)\,dt
\end{equation*}
do not vanish simultaneously, we have \emph{i}., while if both are zero, then we need to go to \emph{ii}. But this process must stop at $J$ since $a_J$ is not zero again by \eqref{RHS}, and we have \emph{iii.}
\end{proof}

Note that a similar result has been obtained by \cite{Harada} for $\gamma=1/2$, in the setting of supercritical and subcritical solutions with respect to the Joseph-Lundgren exponent $p_{JL}$ (corresponding to the stable and unstable cases here respectively).

\medskip

\begin{prop}\label{two-dimensional}
 Let $\kappa$ be as in Proposition \ref{prop:unstable}.
The space of radially symmetric solutions  $\phi$ to problem \eqref{equation:kernel-phi}
that have a bound of the form $|\phi(r)|\leq Cr^{-\frac{n-2\gamma}{2}}$ is at most two-dimensional.
\end{prop}

\begin{proof}
If we set $w=r^{\frac{n-2\gamma}{2}}\phi$ as above, then it satisfies $w=O(1)$ as $t\to \pm\infty$ and it is a solution to $\mathcal L_0 w=0$.

First we look at the indicial roots at $t\to-\infty$, these come from studying the problem $P_\gamma^{(0)}w=0$, this is, $\kappa=0$, and are given in \eqref{lem:num}. From this point of view, we have
\[\begin{split}
P_\gamma^{(0)}w=\mathcal{V_*}(t)w=:h_0=
\begin{cases}
O(1)
	& \text{ as } t\to+\infty\\
O(e^{-q_1|t|})
	& \text{ as } t\to-\infty.
\end{cases}
\end{split}\]
Using Theorem \ref{thm:Hardy-potential}, we obtain a particular solution $w_p^-$ (which equals $w$ since any kernel element of $P_\gamma^{(0)}$ grows as either $t\to\pm\infty$), satisfying
\begin{equation*}
w(t)=w_p^-(t)=O(e^{-q_1|t|})
	\quad \text{ as } t\to-\infty.
\end{equation*}
Using the extra piece of information, we now invert $P_\gamma^{(0)}-\kappa$ using Theorem \ref{thm:Hardy-potential} again, with the conditions
\[
(P_\gamma^{(0)}-\kappa)w=(\mathcal{V_*}(t)-\kappa)w=:h=
\begin{cases}
O(e^{-qt})
	& \text{ as } t\to+\infty\\
O(e^{-q_1|t|})
	& \text{ as } t\to-\infty.
\end{cases}
\]
Its particular solution $w_p^+$, according to the argument in Proposition \ref{prop:unstable+}, satisfies either
\[
w_p^+(t)=
\begin{cases}
a_0^1 \cos(\tau_0 t) + a_0^2 \sin(\tau_0 t) + o(1)
	& \text{ as } t\to+\infty\\
O(e^{-\min\{q_1,\varsigma_0\}|t|})
	& \text{ as } t\to-\infty,
\end{cases}
\]
for some $(a_0^1,a_0^2)\neq(0,0)$ (case \emph{i.}), or, even better when $a_0^1=a_0^2=0$, it has an exponential decay of some order $e^{-\sigma_{j_0}t}(a_{j_0}^1\cos(\tau_{j_0}t)+a_{j_0}^2\sin(\tau_{j_0}t))$ as $t\to+\infty$ (case \emph{ii.} or \emph{iii.}). Here $(a_{j_0}^1,a_{j_0}^2)\neq(0,0)$ but $\tau_{j_0}$ can possibly vanish. In any case, we must have again $w=w_p^+$ because we cannot add any exponentially growing kernel elements of $P_\gamma^{(0)}-\kappa$, nor the bounded kernels $\cos(\tau_0 t)$, $\sin(\tau_0 t)$ due to the decay as $t\to-\infty$.

Next, two solutions with the same $(a_0^1,a_0^2)\neq (0,0)$ must have the same asymptotic expansion as $t\to+\infty$ thanks to the Wro\'{n}skian argument above, and thus agree by unique continuation. This shows that the kernel of $\mathcal{L}_0$ is two-dimensional. The same proof applies also to the case where the leading order term is $e^{-\sigma_{j_0}t}(a_{j_0}^1\cos(\tau_{j_0}t)+a_{j_0}^2\sin(\tau_{j_0}t)$ with $\tau_{j_0}\neq0$. In the particular case $\tau_{j_0}=0$, one obtains a one-dimensional kernel for $\mathcal{L}_0$.
\end{proof}

\section{Poho\v{z}aev identities}

Poho\v{z}aev identities for the fractional Laplacian have been considered in \cite{Frank-Lenzmann,RosOton-Serra1,DelaTorre-Hyder-Martinazzi-Sire}, for instance. Based on our study of ODEs with fractional Laplacian, here we derive some (new) Poho\v{z}aev identities for $w$ a radially symmetric solution of
\begin{equation}\label{maineq}
	P_\gamma w-\kappa w= w^{\frac{n+2\gamma}{n-2\gamma}},\quad w=w(t).
\end{equation}
Here $\kappa$ is a real constant.  For later purposes, it will be convenient to replace
\begin{equation*}
-\kappa=\tau-\Lambda_{n,\gamma}
\end{equation*}
This equation appears, for instance, in the study of minimizers for the fractional Caffarelli-Kohn-Nirenberg inequality \cite{Ao-DelaTorre-Gonzalez}.

By Proposition \ref{prop:extension},  problem \eqref{maineq} is equivalent to
\begin{equation}\label{extension-pohozaev}
\left\{\begin{split}
&\partial_\rho(e_1\rho^{1-2\gamma}\partial_\rho W)+e_2\rho^{1-2\gamma}\partial_{tt}W 	=0, \quad \rho\in(0,\rho_0),\ t\in\mathbb R,\\
&-\lim_{\rho\to 0}\rho^{1-2\gamma}\partial_\rho W(\rho,t)+\tau w- w^{\frac{n+2\gamma}{n-2\gamma}}=0 \quad\text{on }\{\rho=0\}.
\end{split}\right.	
\end{equation}

\begin{prop}\label{prop:Pohozaev1}
If $W=W(t,\rho)$ is a solution of \eqref{extension-pohozaev}, then we have the following Poho\v{z}aev identities:
\begin{equation*}
\begin{split}
&	\tau\int w^2\,dt+\iint e_1\rho^{1-2\gamma}(\partial_\rho W)^2\,d\rho dt=\left(\frac{1}{2}+\frac{n-2\gamma}{2n}\right) \int w^{\frac{2n}{n-2\gamma}} \,dt,\\
&\iint e_2\rho^{1-2\gamma}(\partial_t W)^2		\,d\rho dt=\left(\frac{1}{2}-\frac{n-2\gamma}{2n}\right)\int w^{\frac{2n}{n-2\gamma}}\, dt.
\end{split}	
\end{equation*}
\end{prop}

\begin{proof}
Multiply the first equation in \eqref{extension-pohozaev} by $W$ and integrate by parts; we have
\begin{equation}\label{pohoz1}
\iint e_1(\rho)\rho^{1-2\gamma}(\partial_\rho W)^2+e_2(\rho)\rho^{1-2\gamma}(\partial_t W)^2+\tau \int w^2\,dt=\int w^{\frac{2n}{n-2\gamma}}\,dt.	
\end{equation}
Next we multiply the same equation by $t\partial_t W$; one has
\begin{equation*}
0=\iint \partial_\rho (e_1\rho^{1-2\gamma}\partial_\rho W)t\partial_t W +e_2\rho^{1-2\gamma}\partial_{tt}W t\partial_t W\, d\rho dt=:I_1+I_2.	
\end{equation*}
First we consider $I_1$,
\begin{equation*}
\begin{split}
I_1&=-\int \lim_{\rho\to 0}\rho^{1-2\gamma}\partial_\rho W t\partial_t W \,dt+\frac{1}{2}\iint e_1\rho^{1-2\gamma}(\partial_\rho W)^2\,d\rho dt	\\
&=\frac{\tau}{2}\int w^2\,dt-\frac{n-2\gamma}{2n}\int w^{\frac{2n}{n-2\gamma}}\,dt+\frac{1}{2}\iint e_1\rho^{1-2\gamma}(\partial_\rho {W})^2\,d\rho dt.
\end{split}	
\end{equation*}
Here we have used $e(\rho)\to 1$ as $\rho\to 0$, and the second equation in \eqref{extension-pohozaev}.
Similarly, it holds that
\begin{equation*}
	\begin{split}
I_2=-\frac{1}{2}\iint e_2\rho^{1-2\gamma}(\partial_t{W})^2		\,d\rho dt,
	\end{split}
\end{equation*}
so we have
\begin{equation}\label{pohoz2}
\frac{\tau}{2}\int w^2\,dt-\frac{n-2\gamma}{2n}\int w^{\frac{2n}{n-2\gamma}}\,dt+\frac{1}{2}\iint \left[e_1\rho^{1-2\gamma}(\partial_\rho w)^2-e_2\rho^{1-2\gamma}(\partial_t w)^2\right]		\,d\rho dt=0.	
\end{equation}
Combining \eqref{pohoz1} and \eqref{pohoz2}, one proves the claim of the Proposition.
\end{proof}

\medskip

Now we provide another Poho\v{z}aev type identity based on the variation of constants formula. We discuss first the simpler case where all $\tau_j=0$, then the general case. Let $w$ be a solution to \eqref{maineq}. As in Corollary \ref{cor:ODEsystem}, we write $w$ as
\begin{equation*}
w(t)=\sum_{j=0}^\infty c_j w_j(t),
    \quad \text{ where } \quad
w_j(t):=\int_{\mathbb R}e^{-\sigma_j |t-t'|}h(t')\,dt',
\end{equation*}
for $h=w^{\frac{n+2\gamma}{n-2\gamma}}$. Here
$w_j$ is a particular solution to the second order ODE
 \begin{equation}\label{ODE1}
w_j''-\sigma_j^2w_j=-2\sigma_j w^{\frac{n+2\gamma}{n-2\gamma}}.
\end{equation}

\begin{prop}\label{prop:Pohozaev2}
Let $w$ be a (decaying) radially symmetric solution to \eqref{maineq}. Then
\begin{equation*}
\dfrac{1}{2\gamma}
	\sum_{j}\frac{c_j}{\sigma_j}
		\int_{\mathbb{R}}(w_j')^2\,dt
=\dfrac{1}{2(n-\gamma)}
	\sum_j \sigma_jc_j
		\int_{\mathbb{R}} w_j^2\,dt
=\dfrac{1}{n}
	\int_{\mathbb{R}} w^{\frac{2n}{n-2\gamma}}\,dt.
\end{equation*}
\end{prop}

\begin{proof}
We multiply the above equation \eqref{ODE1} by $c_jw_j/\sigma_j$,  sum in $j$ and integrate by parts:
\begin{equation*}
-\sum_{j}\frac{c_j}{\sigma_j}
	\int_{\mathbb{R}} (w_j')^2\,dt
	-\sum_j \sigma_jc_j
		\int_{\mathbb{R}} w_j^2\,dt
=
	-2\int_{\mathbb{R}} w^{\frac{2n}{n-2\gamma}}\,dt.
\end{equation*}
Similarly, multiply the same equation by $2c_j t w_j'/\sigma_j$, sum in $j$ and integrate:
\begin{equation*}
-\sum_{j}\frac{c_j}{\sigma_j}
	\int_{\mathbb{R}} (w_j')^2\,dt
+\sum_j \sigma_jc_j
	\int_{\mathbb{R}} w_j^2\,dt
=
	2\frac{n-2\gamma}{n}
		\int_{\mathbb{R}} w^{\frac{2n}{n-2\gamma}}\,dt.
\end{equation*}
We add and subtract the two equations, to obtain
\begin{equation*}
\sum_{j}\frac{c_j}{\sigma_j}
	\int_{\mathbb{R}}(w_j')^2\,dt
=
	\left[1-\frac{n-2\gamma}{n}\right]
		\int_{\mathbb{R}} w^{\frac{2n}{n-2\gamma}}\,dt,
\end{equation*}
and
\begin{equation*}
\sum_j \sigma_jc_j
	\int_{\mathbb{R}} w_j^2\,dt
=
	\left[1+\frac{n-2\gamma}{n}\right]
		\int_{\mathbb{R}} w^{\frac{2n}{n-2\gamma}}\,dt.
\end{equation*}
This completes the proof.
\end{proof}

\medskip

In the general case, as in Corollary \ref{cor:ODEsystem-complex}, we write
\begin{equation*}
w(t)=
\mathrm{Re\,}\sum_{j=0}^\infty c_j w_j(t),
    \quad \text{ where } \quad
w_j(t):=
\int_{\mathbb R}e^{-(\sigma_j+i\tau_j) |t-t'|}h(t')\,dt',
\end{equation*}
for $h=w^{\frac{n+2\gamma}{n-2\gamma}}$. Then the following complex-valued ODE system is satisfied,
\begin{equation*}
w_j''-(\sigma_j+i\tau_j)^2w_j
=-2(\sigma_j+i\tau_j)
w^{\frac{n+2\gamma}{n-2\gamma}}.
\end{equation*}

\begin{prop}\label{prop:Pohozaev2-complex}
Let $w$ be a (decaying) radially symmetric solution to \eqref{maineq}. Then
\begin{equation*}
\dfrac{1}{2\gamma}
\mathrm{Re\,}
	\sum_{j}\frac{c_j}{\sigma_j}
		\int_{\mathbb{R}}(w_j')^2\,dt
=\dfrac{1}{2(n-\gamma)}
    \mathrm{Re\,}
	\sum_j \sigma_jc_j
		\int_{\mathbb{R}} w_j^2\,dt
=\dfrac{1}{n}
	\int_{\mathbb{R}} w^{\frac{2n}{n-2\gamma}}\,dt.
\end{equation*}
\end{prop}

\begin{proof}
The proof stays almost the same as in Proposition \ref{prop:Pohozaev2}, except that we take real parts upon testing against the corresponding multiple of $w_j$ and $tw_j'$, which yields
\begin{equation*}
-\mathrm{Re\,}
    \sum_{j}\frac{c_j}{\sigma_j+i\tau_j}
	\int_{\mathbb{R}} (w_j')^2\,dt
-\mathrm{Re\,}
    \sum_j (\sigma_j+i\tau_j)c_j
		\int_{\mathbb{R}} w_j^2\,dt
=
	-2\int_{\mathbb{R}} w^{\frac{2n}{n-2\gamma}}\,dt,
\end{equation*}
and
\begin{equation*}
-\mathrm{Re\,}
\sum_{j}\frac{c_j}{\sigma_j+i\tau_j}
	\int_{\mathbb{R}} (w_j')^2\,dt
+\mathrm{Re\,}
\sum_j (\sigma_j+i\tau_j)c_j
	\int_{\mathbb{R}} w_j^2\,dt
=
	2\frac{n-2\gamma}{n}
		\int_{\mathbb{R}} w^{\frac{2n}{n-2\gamma}}\,dt.
\end{equation*}
It suffices to add and subtract in the same way.
\end{proof}

\bigskip

\noindent\textbf{Acknowledgements.}  M. Fontelos is supported by the Spanish government grant MTM2017-89423-P.  A. DelaTorre (partially) and M.d.M. Gonz\'alez are supported by the Spanish government grant  MTM2017-85757-P. The research of J. Wei is supported by NSERC of Canada. H. Chan has received funding from the European Research Council under the Grant Agreement No. 721675 ``Regularity and Stability in Partial Differential Equations (RSPDE)''.


\begin{thebibliography}{10}

\bibitem{Abdellaoui-Medina-Peral-Primo}
B. Abdellaoui, M. Medina, I. Peral, A. Primo.
The effect of the Hardy potential in some Calder\'on-Zygmund properties for the fractional Laplacian. {\em J. Differential Equations}  260  (2016),  no. 11, 8160--8206.



\bibitem{Ao-Chan-Gonzalez-Wei:supercritical}
W. Ao, H. Chan, M. Gonz\'alez, J. Wei.
\newblock Bound state solutions for the supercritical fractional Schr\"odinger equation.
\newblock {\em Nonlinear Anal.},  193 (2020) 111448.

\bibitem{Ao-DelaTorre-Gonzalez-Wei} W. Ao, A. DelaTorre, M. Gonz\'alez and J. Wei. A gluing approach for the fractional Yamabe problem with prescribed isolated singularities. To appear in {\em Journal f\"ur die reine und angewandte Mathematik}.


\bibitem{Ao-Chan-Gonzalez-Wei}
W. Ao, H. Chan, M. Gonz\'alez, J. Wei. Existence of positive weak solutions for fractional Lane-Emden equations with prescribed singular set. {\em Calc. Var. Partial Differential Equations} 57 (2018), no. 6, Art. 149, 25 pp.

\bibitem{Ao-Chan-DelaTorre-Fontelos-Gonzalez-Wei}
W.~Ao, H.~Chan,  A.~Delatorre, M. A.~Fontelos, M.d.M.~Gonz\'alez, J.~Wei.
\newblock On higher dimensional singularities for the fractional Yamabe problem: a non-local Mazzeo-Pacard program.
\newblock {\em Duke Math Journal} 168, n. 17 (2019), 3297-3411.

\bibitem{Ao-DelaTorre-Gonzalez}
W. Ao, A. DelaTorre, M. Gonz\'alez.
\newblock On the fractional Caffarelli-Kohn-Nirenberg inequality.
\newblock In preparation.

\bibitem{Ao-Gonzalez-Sire}
W. Ao, M. Gonz\'alez, Y. Sire.
\newblock Boundary connected sum of Escobar manifolds.
\newblock To appear in {\em Journal of Geometric Analysis}.

\bibitem{Beckner:Pitts}
W. Beckner.
\newblock Pitt's inequality and the uncertainty principle.
\newblock {\em Proc. Amer. Math. Soc.} 123(6):1897--1905, 1995.

\bibitem{Cabre-Sire1}
X.~Cabr\'{e}, Y.~Sire. Nonlinear equations for fractional Laplacians, I: Regularity, maximum principles, and Hamiltonian estimates.
{\em Ann. Inst. H. Poincar\'{e} Anal. Non Lin\'{e}aire} 31 (2014), no. 1, 23--53.

\bibitem{CabreSola-Morales}
  X. Cabr{\'e}, J. Sol{\`a}-Morales.
     \newblock Layer solutions in a half-space for boundary reactions.
     \newblock {\em Communications on Pure and Applied Mathematics}, 58 (2005), 1678--1732.

\bibitem{Caffarelli-Jin-Sire-Xiong}
L.~Caffarelli, T.~Jin, Y.~Sire, J.~Xiong.
\newblock Local analysis of solutions of fractional semi-linear elliptic
  equations with isolated singularities.
\newblock {\em Arch. Ration. Mech. Anal.}, 213 (2014), n. 1, 245--268.


\bibitem{Caffarelli-Silvestre}
L. Caffarelli, L. Silvestre. An extension problem related to the fractional Laplacian. {\em Comm. Partial Differential Equations} 32 (2007) 1245--1260.

\bibitem{Calvez-Carrillo-Hoffmann}
V. Calvez, J. A. Carrillo, F. Hoffmann. Uniqueness of stationary states for singular Keller-Segel type models. Preprint arXiv:1905.07788.

\bibitem{Case-Chang}
J.~Case, S.-Y. A.~Chang.
\newblock On fractional {G}{J}{M}{S} operators.
\newblock {\em  Communications on Pure and Applied Mathematics} 69 (2016), no. 6, 1017--1061.

\bibitem{Chan-DelaTorre}
H. Chan, A. DelaTorre.
Singular solutions and a critical Yamabe problem revisited. Preprint arXiv:1912.10352.

\bibitem{Chan-Gonzalez-Huang-Mainini-Volzone} H. Chan, M.d.M. Gonz\'alez, Y. Huang, E. Mainini and B. Volzone. Uniqueness of entire ground states for the fractional plasma problem. Preprint 	arXiv:2003.01093.






\bibitem{Chang-Gonzalez}
S.-Y. A.~Chang, M.~Gonz{\'a}lez.
\newblock Fractional {L}aplacian in conformal geometry.
\newblock {\em Adv. Math.}, 226 (2011), no. 2, 1410--1432.

\bibitem{Chen-Quaas}
H. Chen, A. Quaas.
Classification of isolated singularities of nonnegative solutions to fractional semi-linear elliptic equations and the existence results.
{\em J. Lond. Math. Soc.} (2) 97 (2018), no. 2, 196--221.


\bibitem{Chen-Frank-Weth}
S. Chen, R. Frank, T. Weth.
Remainder terms in the fractional Sobolev inequality.
{\em Indiana Univ. Math. J.} 62 (2013), no. 4, 1381--1397.

\bibitem{dps} J. D\'avila, M. del Pino, Y. Sire. Non degeneracy of the bubble in the critical case for non local equations. {\em Proc. Amer. Math. Soc.} 141 (2013), 3865--3870.

\bibitem{Davila-DelPino-Wei}
J. D\'avila, M. del Pino, J. Wei.
Concentrating standing waves for the fractional nonlinear Schr\"odinger equation.
{\em J. Differential Equations}  256  (2014),  no. 2, 858--892.

\bibitem{DelaTorre-Gonzalez}
A.~DelaTorre, M.d.M.~Gonz{\'a}lez.
\newblock Isolated singularities for a semilinear equation for the fractional
  {L}aplacian arising in conformal geometry.
\newblock {\em Rev. Mat. Iber.} 34 (2018), no. 4, 1645--1678.

\bibitem{DelaTorre-delPino-Gonzalez-Wei}
A. DelaTorre, M. del Pino, M.d.M. Gonz\'alez, J. Wei.  Delaunay-type singular solutions for the fractional {Y}amabe problem.
\newblock{\em Math Annalen} 369 (2017) 597--62.


\bibitem{DelaTorre-Hyder-Martinazzi-Sire}
A. DelaTorre, A. Hyder, L. Martinazzi, Y. Sire. The nonlocal mean-field equation on an interval.
\newblock{\em Communications in Contemporary Mathematics} (2019) 1950028 (19 pages).

\bibitem{Felmer-Quaas-Tan}
P. Felmer, A. Quaas, J. Tan,
\newblock Positive solutions of the nonlinear Schrdinger equation with the fractional {L}aplacian.
\newblock {\em Proc. Roy. Soc. Edinburgh Sect. A} 142 (2012), no. 6, 12371262.




\bibitem{Fall-Felli}
M. Fall, V. Felli. Unique continuation property and local asymptotics of solutions to fractional elliptic equations. {\em Comm. Partial Differential Equations}  39  (2014),  no. 2, 354--397.

\bibitem{Frank:uniqueness}
R. Frank.
On the uniqueness of ground states of non-local equations.
 {\em Journ\'ees \'Equations aux D\'eriv\'ees Partielles} 2011, Exp. No. V, 10 pp.

\bibitem{Frank-Lenzmann}
R. Frank, E. Lenzmann.
Uniqueness of non-linear ground states for fractional Laplacians in $\mathbb R$. {\em Acta Math}.  210  (2013),  no. 2, 261--318.

\bibitem{Frank-Lenzmann-Silvestre}
R. Frank, E. Lenzmann, L. Silvestre.
Uniqueness of radial solutions for the fractional Laplacian.
{\em Comm. Pure Appl. Math.}  69  (2016),  no. 9, 1671--1726.

\bibitem{Frank-Lieb-Seiringer:Hardy}
R.~L. Frank, E. Lieb and R. Seiringer.
\newblock Hardy-Lieb-Thirring inequalities for fractional {S}chr\"odinger operators.
\newblock {\em J. Amer. Math. Soc.} 21(4):925--950, 2008.

\bibitem{Frank-Merz-Siedentop}
R. Frank, K.Merz, H. Siedentop. Equivalence of Sobolev norms involving generalized Hardy operators.
Preprint arXiv:1807.09027.

\bibitem{Frank-Merz-Siedentop-Simon}
R. Frank, K.Merz, H. Siedentop, B. Simon.
Proof of the strong Scott conjecture for Chandrasekhar atoms.
Preprint arXiv:1907.04894

\bibitem{Gonzalez:survey}
M.~d.~M. Gonz{\'a}lez.
Recent progress on the fractional Laplacian in conformal geometry.
Chapter in {\em Recent Developments in Nonlocal Theory}. Berlin, Boston: Sciendo Migration.  G. Palatucci \& T. Kuusi (Eds.) (2018)

\bibitem{Gonzalez-Mazzeo-Sire}
M.~d.~M. Gonz{\'a}lez, R.~Mazzeo, Y.~Sire.
\newblock Singular solutions of fractional order conformal {L}aplacians.
\newblock {\em J. Geom. Anal.}, 22 (2012), no. 3, 845--863.

\bibitem{Gonzalez-Qing}
M.~d.~M. Gonz{\'a}lez, J.~Qing.
\newblock Fractional conformal {L}aplacians and fractional {Y}amabe problems.
\newblock {\em Anal. PDE}, 6 (2013), no. 7, 1535--1576.

\bibitem{Gonzalez-Wang}
M.~d.~M. Gonz{\'a}lez, M.~Wang.
\newblock Further results on the fractional yamabe problem: the umbilic case.
\newblock {\em J. Geom. Anal.} 28 (2018), no. 1, 22--60.

\bibitem{Graham-Zworski}
C.~R. Graham and M.~Zworski.
\newblock Scattering matrix in conformal geometry.
\newblock {\em Invent. Math.}, 152(1):89--118, 2003.




\bibitem{Harada}
J. Harada.
\newblock Positive solutions to the Laplace equation with nonlinear boundary conditions on the half space.
\newblock {\em Calc. Var. Partial Differential Equations} 50 (2014), no. 1--2, 399--435.




\bibitem{Herbst}
I. Herbst. Spectral theory of the operator $(p^2+m^2)^{1/2}-Ze^2/r$. {\em Comm. Math. Phys}. 53 (1977), no. 3, 285--294.

\bibitem{Kim-Musso-Wei} S. Kim, M. Musso, J. Wei.  Existence theorems of the fractional Yamabe problem.  {\em Anal. \& PDE} 11 (2018), no. 1, 75--113.

\bibitem{Kovalenko-Perelmuter-Semenov} V. Kovalenko, M. Perelmuter, Y. Semenov. Schr\"odinger operators with $L^{1/2}_W(\mathbf R^l)$-potentials. {\em J. Math. Phys.} 22 (1981), no. 5, 1033--1044.

\bibitem{Prehistory-Hardy}
A. Kufner, L. Maligranda, L. Persson.
The prehistory of the Hardy inequality.
{\em Amer. Math. Monthly}  113  (2006),  no. 8, 715--732.

\bibitem{Luo-Wei-Zou}
S. Luo, J. Wei, W. Zou.
On a transcendental equation involving quotients of gamma functions.
{\em Proc. Amer. Math. Soc.} 145 (2017), no. 6, 2623--2637.

\bibitem{Mayer-Ndiaye} M. Mayer, C. B. Ndiaye. Fractional Yamabe problem on locally flat conformal infinities of Poincare-Einstein manifolds. Preprint.

\bibitem{Mazzeo:edge} R. Mazzeo. Elliptic theory of differential edge operators {I}. {\em Comm. Partial Differential Equations} 16 (1991), no. 10, 1615--1664.

\bibitem{Mazzeo:edge2} R. Mazzeo, B. Vertman. Elliptic theory of differential edge operators, II: Boundary value problems.
{\em Indiana Univ. Math. J.}  63  (2014),  no. 6, 1911--1955.



\bibitem{RosOton-Serra1}
X. Ros-Oton, J. Serra.
The Pohozaev identity for the fractional Laplacian.
{\em Arch. Rat. Mech. Anal.} 213 (2014), 587--628.



\bibitem{Ruland}
A. R\"uland. Unique continuation for fractional Schr\"odinger equations with rough potentials.
{\em Comm. Partial Differential Equations}  40  (2015),  no. 1, 77--114.




\bibitem{Stinga-Torrea}
P. Stinga, J. Torrea.
\newblock Extension problem and Harnack's inequality for some fractional operators.
\newblock {\em Comm. Partial Differential Equations} 35 (2010), no. 11, 2092--2122.

\bibitem{Yafaev}
D. Yafaev.
Sharp constants in the Hardy--Rellich Inequalities.
{\em J. Funct. Anal.} 168 (1999), no. 1, 121--144.


\end{thebibliography}
\end{document}